\documentclass[10pt]{article}
   \usepackage{graphicx}
   \usepackage{epsfig}
   \usepackage{amssymb}
   \usepackage{amsmath}
   \usepackage{amsthm}
   \newtheorem{lemma}{Lemma}[section]
   \newtheorem{theorem}{Theorem}[section]

   {\end{description}}

   {\hfill $\bullet$ \\}

   \newcommand{\be}{\begin{equation}}
   \newcommand{\ee}{\end{equation}}

\begin{document}
    \title{An unconditionally stable numerical approach for solving a nonlinear distributed delay Sobolev model}
  \author{Eric Ngondiep\thanks{\textbf{Correspondence to:} Eric Ngondiep, engondiep@imamu.edu.sa/ericngondiep@gmail.com}}
   \date{\small{Department of Mathematics and Statistics, College of Science, Imam Mohammad Ibn Saud\\ Islamic University
        (IMSIU), $90950$ Riyadh $11632,$ Saudi Arabia.}}
   \maketitle

   \textbf{Abstract.}
   This paper proposes an unconditionally stable numerical method for solving a nonlinear Sobolev model with distributed delay. The proposed computational approach approximates the time derivative by interpolation technique whereas the spatial derivatives are approximated using the finite element approximation. This combination is simple and easy to implement. Both stability and error estimates of the constructed method are deeply analyzed in a strong norm which is equivalent to the $H^{1}$-norm. The theoretical results indicate that the constructed approach is unconditionally stable, spatial fourth-order accurate, second-order convergent in time and more efficient than a large class of numerical methods discussed in the literature for solving a general class of delay Sobolev problems. Some numerical examples are carried out to confirm the theory and demonstrate the applicability and validity of the developed technique.\\
    \text{\,}\\

   \ \noindent {\bf Keywords:} distributed delay, nonlinear Sobolev equation, a strong norm, high-order computational technique, unconditional stability and error estimates.\\
   \\
   {\bf AMS Subject Classification (MSC). 65M12, 65M06}.

  \section{Introduction}\label{sec1}
    Modeling in physical, social and biological sciences sometimes considers the time delay inherent in the phenomena \cite{35xg,4xg,31xg}. In the literature, several aspects of distributed delay models have been analyzed such as: oscillation \cite{29xg}, stability properties \cite{16xg}, periodic solutions \cite{22xg,19xg,13xg}, etc... Sobolev equations were used to model complex real-worth entities in physics and engineering: thermodynamics \cite{7tr}, shear in second-order fluid \cite{2tr,4tr}, flows of fluids through fissured \cite{1tr}, and so on. These equations are deeply analyzed in the literature by a large class of numerical methods \cite{14tr,16tr,10tr,12tr}. Most recently, the authors \cite{19tr,17tr} have shown that Sobolev equations with memory, also called delay Sobolev equations (DSEs) are more realistic and effective in the modeling of practical problems. A broad range of applications have motivated many researchers to develop efficient numerical approaches in approximate solutions such as: finite difference schemes \cite{21tr,22tr}, linearized compact difference methods \cite{27tr,25tr,29tr,33tr} and compact multistep procedures \cite{32tr,31tr}.\\

    In this paper, we consider the nonlinear distributed delay Sobolev model defined in \cite{tr} as
     \begin{equation}\label{1}
     (\mathcal{I}-\beta\Delta)v_{t}-\alpha\Delta v=f\left(x,t,v,\int_{t-\tau}^{t}g(x,t,s,v(s))ds\right), \text{\,\,\,\,on\,\,\,\,} \Omega\times[0,\text{\,}T_{f}),
     \end{equation}
      subjects to initial condition
      \begin{equation}\label{2}
      v(x,t)=u_{0}(x,t),\text{\,\,\,\,for\,\,\,\,}(x,t)\in(\overline{\Omega}=\Gamma\cup\Omega)\times[-\tau,\text{\,}0],
     \end{equation}
     and boundary condition
      \begin{equation}\label{3}
      v(x,t)=\rho(x,t)\text{\,\,\,\,for\,\,\,\,}(x,t)\in\Gamma\times[0,\text{\,}T_{f}],
     \end{equation}
     where $\mathcal{I}$ means the identity operator, $\Omega$ is a bounded domain in $\mathbb{R}^{d}$ ($d=1$, $2$), and $\Gamma$ denotes its boundary, $T_{f}>0$ represents the final time, $\tau>0$ is a given constant, $\alpha$ and $\beta$ are two positive physical parameters, $v$ is the unknown function, $v_{t}$ means $\frac{\partial v}{\partial t}$, $\Delta$ designates the laplacian operator, while $u_{0}$ and $\rho$ are the initial and boundary conditions, respectively. To ensure the existence and uniqueness of a smooth solution to the initial-boundary value problem $(\ref{1})$-$(\ref{3})$, $u_{0}$ and $\rho$ are assumed to satisfy the requirement $u_{0}(x,0)=\rho(x,0)$, for every $x\in\Gamma$. Additionally, for the sake of stability analysis and error estimates, we assume that both functions $f$ and $g$ are regular enough and satisfy the local Lipschitz conditions in the third and fourth variables. It is worth mentioning that equation $(\ref{1})$ falls in the class of complex unsteady partial differential equations (PDEs) \cite{en1,17tr,21xg,en2,19tr,18tr} whose the computation of exact solutions is very difficult and often impossible. When the function $g$ equals zero, this equation yields a nonlinear Sobolev equation while for $\beta=0$ and $g=0$, equation  $(\ref{1})$ becomes a nonlinear reaction-diffusion equation. Many authors have proposed efficient numerical methods in computed solutions for both classes of time-dependent PDEs. For more details, interested readers can consult \cite{en3,36tr,15tr,en4,27tr} and the references therein. Most recently, some authors \cite{tr,33tr} have developed computational approaches for solving Sobolev and semilinear parabolic problems with distributed delay. In this paper, we construct a high-order unconditionally stable computational technique in a computed solution of the parabolic equation  $(\ref{1})$ subjects to initial condition  $(\ref{2})$ and boundary condition  $(\ref{3})$. The proposed approach consists to approximate the time derivative using a "noncentered" three-level difference scheme whereas the approximation of the space derivatives considers the finite element formulation. The new algorithm is unconditionally stable, second-order accurate in time and spatial fourth-order convergent. In addition, the first two terms required to start the algorithm are directly obtained from the initial condition. It is also important to remind that the new strategy is more efficient than a wide set of numerical schemes discussed in the literature \cite{11tr,15tr,tr,21tr,33tr,25tr}, for solving a general class of delay PDEs. The highlights of the paper are the following:
     \begin{description}
      \item[(i)] full description of the high-order computational method in an approximate solution of the initial-boundary value problem $(\ref{1})$-$(\ref{3})$,
      \item[(ii)] analysis of stability and error estimates of the constructed approach,
      \item[(iii)] numerical experiments to confirm the theory and demonstrate the applicability and validity of the proposed algorithm.
     \end{description}
      The remainder of the paper is organized as follows. In Section $\ref{sec2}$, we construct the high-order computational technique for solving the nonlinear unsteady equation $(\ref{1})$ with initial and boundary conditions $(\ref{2})$ and $(\ref{3})$, respectively. Section $\ref{sec3}$ deals with the stability analysis and error estimates of the developed algorithm whereas some numerical examples which confirm the theoretical studies are carried out in Section $\ref{sec4}$. Finally, in Section $\ref{sec5}$ we draw the general conclusions and provide our future works.

    \section{Description of the proposed computational technique}\label{sec2}
    In this Section, we describe the proposed high-order numerical approach in an approximate solution of the nonlinear Sobolev equation with distributed delay $(\ref{1})$ subjects to initial and boundary conditions $(\ref{2})$ and $(\ref{3})$, respectively.\\

     Let $m$ and $N$ be two positive integers. Set $\sigma=\frac{\tau}{m}$, be the time step. For the convenience of discretization we assume that the final time $T_{f}$ is chosen so that $\sigma=\frac{T_{f}}{N}$. Consider $\Omega_{\sigma}=\{t_{n}=n\sigma,\text{\,}-m\leq n\leq N\}$, be a uniform partition of the time interval $[-\tau,\text{\,}T_{f}]$. Suppose that $\mathcal{T}_{h}$ is the regular partition of the domain $\overline{\Omega}=\Gamma\cup\Omega$, consisting of triangles (or segments) $T$ of diameter (or length) $d(T)$, where $h=\sup\{d(T),\text{\,\,}T\in\mathcal{T}_{h}\}$, satisfying the following restrictions: (a) the triangulation $\mathcal{T}_{h}$ is regular while the triangulation $\mathcal{T}_{\Gamma,h}$ induced on the boundary $\Gamma$ is quasi-uniform, (b) the interior of any triangle $T$ is nonempty and the intersection of the interior of two different triangles is empty whereas the intersection of two triangles is empty, or a common vertex or edge.\\

     We denote $H^{q}(\Omega)$, be the Sobolev space equipped with the scalar product $\left(\cdot,\cdot\right)_{q}$ and the corresponding norm is $\left\|\cdot\right\|_{q}$, where $q$ is a nonnegative integer. Additionally, the space $L^{2}(\Gamma)$ is endowed with the inner product $\left(\cdot,\cdot\right)_{\Gamma}$ and associated norm is $\left\|\cdot\right\|_{\Gamma}$. We remind that the space $H^{0}(Q)$ coincides with $L^{2}(Q)$, where $Q=\Omega$ or $\Gamma$. Since the functions $f$ and $g$ are assumed to be smooth enough, the space containing the analytical solution of the nonlinear distributed delay Sobolev equation $(\ref{1})$ with boundary condition $(\ref{3})$ should be defined as
     \begin{equation}\label{4}
       V=\{u\in H^{4}(0,T_{f};\text{\,}H^{5}(\Omega)):\text{\,\,}u|_{\Gamma}=\rho\}.
       \end{equation}
     In addition, the finite element space approximating the solution of the parabolic equation $(\ref{1})$ is given by
     \begin{equation}\label{5}
       V_{h}=\{u_{h}(t)\in H^{1}(\Omega):\text{\,\,}u_{h}(t)|_{T}\in\mathcal{P}_{5}(T),\text{\,\,}\forall T\in\mathcal{T}_{h},\text{\,\,for\,\,}t\in[0,\text{\,}T_{f}]\},
       \end{equation}
       where, $\mathcal{P}_{5}(T)$ means the set of polynomials defined on $T$ with degree less than or equal $5$. Let $l=(l_{1},...,l_{d})\in\mathbb{N}^{d}$, $x=(x_{1},...,x_{d})\in\mathbb{R}^{d}$, $|l|=l_{1}+...+l_{d}$, $\partial x^{|l|}=\partial x_{1}^{l_{1}}...\partial x_{d}^{l_{d}}$, $D_{x}^{l}u=\frac{\partial^{|l|}u}{\partial x^{|l|}}$ and $D_{t,x}^{s,l}u=\frac{\partial^{s+|l|}u}{\partial t^{s}\partial x^{|l|}}$, where $D_{t,x}^{0,0}u:=u$ and $D_{x}^{0}u:=u$. Suppose that $p,q\in\mathbb{N}$, the Sobolev spaces $H^{q}(\Omega)$, $L^{\infty}(0,T_{f};\text{\,}H^{q})$ and $H^{p}(0,T;\text{\,}H^{q})$, are defined as
    \begin{equation}\label{6}
     H^{q}(\Omega)=\left\{u\in L^{2}(\Omega):\text{\,\,}D_{x}^{l}u\in L^{2}(\Omega),\text{\,\,for\,\,}|l|\leq q\right\},
       \end{equation}
     \begin{equation}\label{7}
     L^{\infty}(0,T_{f};H^{q})=\left\{u\in L^{\infty}(0,T_{f};\text{\,}L^{2}):\text{\,\,}D_{x}^{l}u\in L^{\infty}(0,T_{f};\text{\,}L^{2}),\text{\,\,for\,\,}|l|\leq q\right\},
       \end{equation}
      \begin{equation}\label{8}
     H^{p}(0,T_{f};\text{\,}H^{q})=\left\{u\in L^{2}(0,T_{f};\text{\,}\text{\,}L^{2}):\text{\,\,}D_{t,x}^{s,l}u\in L^{2}(0,T_{f};\text{\,}L^{2}),\text{\,\,}s=0,1,...,p,\text{\,\,}|l|\leq q\right\}.
     \end{equation}
     Here: the spaces $L^{2}(\Omega)$ and $[L^{2}(\Omega)]^{2}$ are endowed with the scalar products $\left(\cdot,\cdot\right)_{0}$ and $\left(\cdot,\cdot\right)_{\bar{0}}$, and associated norms $\left\|\cdot\right\|_{0}$ and $\left\|\cdot\right\|_{\bar{0}}$, respectively.
     \begin{equation*}
      \left(u_{1},u_{2}\right)_{0}=\int_{\Omega}u_{1}u_{2}\text{\,}dx,\text{\,\,}\left\|u_{1}\right\|_{0}=\sqrt{\left(u_{1},u_{1}\right)_{0}},\text{\,\,\,\,\,}\forall u_{1},u_{2}\in L^{2}(\Omega),\text{\,\,and\,\,}
       \end{equation*}
       \begin{equation}\label{9}
        \left(z,w\right)_{\bar{0}}=\int_{\Omega}z^{t}w\text{\,}dx,\text{\,\,}\left\|w\right\|_{\bar{0}}=\sqrt{\left(w,w\right)_{\bar{0}}},\text{\,\,\,\,\,for\,\,}z,w\in [L^{2}(\Omega)]^{2},
       \end{equation}
       where $z^{t}$ denotes the transpose of a vector-function $z\in[L^{2}(\Omega)]^{2}$. Furthermore, the Sobolev spaces: $H^{q}(\Omega)$, $L^{\infty}(0,T_{f};\text{\,}H^{q})$ and $H^{p}(0,T_{f};\text{\,}H^{q})$, are equipped with the following norms
      \begin{equation*}
       \|u\|_{q}=\left(\underset{|m|\leq q}{\sum}\|D_{x}^{m}u\|_{0}^{2}\right)^{\frac{1}{2}},\text{\,\,\,\,}\forall u\in H^{q}(\Omega),\text{\,\,\,\,}\||w|\|_{q,\infty}=\underset{t\in[0,\text{\,}T_{f}]}{\max}\|w(t)\|_{q},\text{\,\,\,}\forall w\in L^{\infty}(0,T_{f};\text{\,}H^{q}),
       \end{equation*}
       \begin{equation}\label{9a}
        \||w|\|_{q,p}=\left(\underset{l=0}{\overset{p}\sum}\int_{0}^{T_{f}}\underset{|m|\leq q}{\sum}\left\|D_{t,x}^{l,m}w(t)\right\|_{0}^{2}dt\right)^{\frac{1}{2}}, \text{\,\,\,\,\,}\forall w\in H^{p}(0,T_{f};\text{\,}H^{q}).
       \end{equation}

        We introduce the bilinear operators $a(\cdot,\cdot)$ and $a_{\Gamma}(\cdot,\cdot)$ defined as:
       \begin{equation}\label{10}
       a(u,w)=\left(\nabla u,\nabla w\right)_{\bar{0}},\text{\,\,\,}a_{\Gamma}(u,w)=\left(u,w\right)_{\Gamma}=\int_{\Gamma}u(\nabla w)^{t}\vec{z}\text{\,}d\Gamma,
       \end{equation}
       where $\vec{z}$ represents the unit outward normal vector on $\Gamma=\partial\Omega$. The following integration by parts will play an important role in our study.
       \begin{equation}\label{11}
       \left(\Delta u,w\right)_{0}=a_{\Gamma}(u,w)-a(u,w),\text{\,\,\,for\,\,\,every\,\,\,}u\in H^{2}(\Omega)\text{\,\,\,and\,\,\,}w\in H^{1}(\Omega).
       \end{equation}

       Finally, for the convenience of writing, we set  $u^{n}=u(t_{n})$ and $u:=u(t)\in H^{1}(\Omega)$, for any $t\in[0,\text{\,}T_{f}]$. Now, applying equation $(\ref{1})$ at the discrete time $t_{n+1}$, for $n=0,1,2,...,N-1$, and using the fact that $\tau=t_{m}$, this gives
       \begin{equation}\label{12}
       (\mathcal{I}-\beta\Delta)v_{t}^{n+1}-\alpha\Delta v^{n+1}=f\left(x,t_{n+1},v^{n+1},\int_{t_{n-m+1}}^{t_{n+1}}g(x,t_{n+1},s,v(s))ds\right).
       \end{equation}

       Approximating the function $v(t)$ at the discrete points: $(t_{n-1},\text{\,}v^{n-1})$, $(t_{n},\text{\,}v^{n})$ and $(t_{n+1},\text{\,}v^{n+1})$, this yields
       \begin{equation*}
        v(t)=\frac{(t-t_{n})(t-t_{n+1})}{(t_{n-1}-t_{n})(t_{n-1}-t_{n+1})}v^{n-1}+\frac{(t-t_{n-1})(t-t_{n+1})}{(t_{n}-t_{n-1})(t_{n}-t_{n+1})}v^{n}+
        \frac{(t-t_{n-1})(t-t_{n})}{(t_{n+1}-t_{n-1})(t_{n+1}-t_{n})}v^{n+1}+
       \end{equation*}
       \begin{equation*}
        \frac{1}{6}(t-t_{n-1})(t-t_{n})(t-t_{n+1})v_{3t}(\epsilon_{1}(t))=\frac{1}{2\sigma^{2}}\left[(t^{2}-(t_{n}+t_{n-1})t+t_{n}t_{n-1})v^{n+1}-\right.
       \end{equation*}
       \begin{equation*}
       \left.2(t^{2}-(t_{n-1}+t_{n+1})t+t_{n+1}t_{n-1})v^{n}+(t^{2}-(t_{n}+t_{n+1})t+t_{n}t_{n+1})v^{n-1}\right] +\frac{1}{6}(t-t_{n-1})(t-t_{n})(t-t_{n+1})v_{3t}(\epsilon_{1}(t)),
       \end{equation*}
       where $\epsilon_{1}(t)$ is between the maximum and the minimum of $t_{n-1}$, $t_{n}$, $t_{n+1}$ and $t$. The time derivative of this equation provides
       \begin{equation*}
        v_{t}(t)=\frac{1}{2\sigma^{2}}\left[(2t-t_{n}-t_{n-1})v^{n+1}-2(2t-t_{n-1}-t_{n+1})v^{n}+(2t-t_{n}-t_{n+1})v^{n-1}\right]+
       \end{equation*}
       \begin{equation*}
        \frac{1}{6}[(t-t_{n-1})(t-t_{n})(t-t_{n+1})\frac{\partial}{\partial t}v_{3t}(\epsilon_{1}(t))+v_{3t}(\epsilon_{1}(t))\frac{d}{dt}((t-t_{n-1})(t-t_{n})(t-t_{n+1}))].
       \end{equation*}

       Performing direct calculations, it holds
       \begin{equation*}
        v_{t}^{n+1}=\frac{1}{2\sigma}(3v^{n+1}-4v^{n}+v^{n-1})+\frac{\sigma^{2}}{3}v_{3t}(\epsilon_{1}(t)).
       \end{equation*}

       Substituting this equation into $(\ref{12})$ and rearranging terms, this results in
       \begin{equation*}
       \frac{1}{2\sigma}(\mathcal{I}-\beta\Delta)(3v^{n+1}-4v^{n}+v^{n-1})-\alpha\Delta v^{n+1}= f\left(x,t_{n+1},v^{n+1},\int_{t_{n-m+1}}^{t_{n+1}}g(x,t_{n+1},s,v(s))ds\right)-
       \end{equation*}
       \begin{equation}\label{13}
        \frac{\sigma^{2}}{3}(\mathcal{I}-\beta\Delta)v_{3t}(\epsilon_{1}(t)).
       \end{equation}

       Multiplying both sides of equation $(\ref{13})$ by $2\sigma$, for any $u\in H^{1}(\Omega)$, using the inner product $\left(\cdot,\cdot\right)_{0}$, defined in relation $(\ref{9})$, and rearranging terms, we obtain
       \begin{equation}\label{14}
       \left((\mathcal{I}-\beta\Delta)(3v^{n+1}-4v^{n}+v^{n-1}),u\right)_{0}-2\alpha\sigma\left(\Delta v^{n+1},u\right)_{0}=2\sigma\left(G(v^{n+1}),u\right)_{0}-
        \frac{2\sigma^{3}}{3}\left((\mathcal{I}-\beta\Delta)v_{3t}(\epsilon_{1}(t)),u\right)_{0},
       \end{equation}
       where
       \begin{equation}\label{15}
        G(v^{n+1})=f\left(x,t_{n+1},v^{n+1},\int_{t_{n-m+1}}^{t_{n+1}}g(x,t_{n+1},s,v(s))ds\right).
       \end{equation}

        Utilizing the integration by parts given by equation $(\ref{11})$, it is not difficult to observe that equation $(\ref{14})$ is equivalent to
       \begin{equation*}
        \left(3v^{n+1}-4v^{n}+v^{n-1},u\right)_{0}+\beta a(3v^{n+1}-4v^{n}+v^{n-1},u)+2\alpha\sigma a(v^{n+1},u)-\beta a_{\Gamma}(3v^{n+1}-4v^{n}+v^{n-1},u)-
       \end{equation*}
       \begin{equation}\label{16}
       2\alpha\sigma a_{\Gamma}(v^{n+1},u)=2\sigma\left(G(v^{n+1}),u\right)_{0}-\frac{2\sigma^{3}}{3}\left((\mathcal{I}-\beta\Delta)v_{3t}(\epsilon_{1}(t)),u\right)_{0}.
       \end{equation}

       Let $A(\cdot,\cdot)$ be the bilinear operator defined as
       \begin{equation}\label{17}
       A(u,w)=\left(u,w\right)_{0}+\beta a(u,w),\text{\,\,\,\,\,\,\,}\forall u,w\in H^{1}(\Omega).
       \end{equation}

       Using the operator $A(\cdot,\cdot)$ together with the boundary condition $(\ref{3})$, that is, $v^{r}=\rho^{r}$, on $\Gamma\times[0,\text{\,}T_{f}]$, for $r=n-1,n,n+1$, it is easy to see that equation $(\ref{16})$ can be written as
       \begin{equation*}
        3A(v^{n+1},u)+2\alpha\sigma a(v^{n+1},u)=A(4v^{n}-v^{n-1},u)+\beta a_{\Gamma}(3\rho^{n+1}-4\rho^{n}+\rho^{n-1},u)+2\alpha\sigma a_{\Gamma}(\rho^{n+1},u)+
       \end{equation*}
       \begin{equation}\label{18}
       2\sigma\left(G(v^{n+1}),u\right)_{0}-\frac{2\sigma^{3}}{3}\left((\mathcal{I}-\beta\Delta)v_{3t}(\epsilon_{1}(t)),u\right)_{0}.
       \end{equation}

       Since the term $G(v^{n+1})=f\left(x,t_{n+1},v^{n+1},\int_{t_{n-m+1}}^{t_{n+1}}g(x,t_{n+1},s,v(s))ds\right)$ seems to be too complex, so solve equation $(\ref{18})$ should be time consuming. To overcome this issue, we must approximate the quantity $G(v^{n+1})$. The term $\int_{t_{n-m+1}}^{t_{n+1}}g(x,t_{n+1},s,v(s))ds$ can be decomposed as
       \begin{equation}\label{18a}
       \int_{t_{n-m+1}}^{t_{n+1}}g(x,t_{n+1},s,v(s))ds=\overset{-1}{\underset{j=-m}\sum}\int_{t_{n+1+j}}^{t_{n+2+j}}g(x,t_{n+1},s,v(s))ds.
       \end{equation}

       In addition, the interpolation of the function $g(x,t_{n+1},\cdot,v(\cdot))$ at the points $t_{n+1+j}$ and $t_{n+2+j}$, for $j=-m,-m+1,...,-1$, gives
       \begin{equation*}
        g(x,t_{n+1},s,v(s))=\frac{s-t_{n+1+j}}{\sigma}g(x,t_{n+1},t_{n+2+j},v^{n+2+j})-\frac{s-t_{n+2+j}}{\sigma}g(x,t_{n+1},t_{n+1+j},v^{n+1+j})+
       \end{equation*}
       \begin{equation}\label{19}
       \frac{1}{2}(s-t_{n+1+j})(s-t_{n+2+j})\frac{\partial^{2}}{\partial s^{2}}[g(x,t_{n+1},\epsilon_{2}(s),v(\epsilon_{2}(s)))],
       \end{equation}
       where $\epsilon_{2}(s)$ is between the minimum and maximum of $t_{n+1+j}$, $t_{n+2+j}$ and $s$. Using equation $(\ref{19})$ the integration of $g(x,t_{n+1},s,v(s))$ over the interval $[t_{n+1+j},\text{\,}t_{n+2+j}]$ yields
       \begin{equation*}
        \int_{t_{n+1+j}}^{t_{n+2+j}}g(x,t_{n+1},s,v(s))ds=\frac{1}{\sigma}\left[g(x,t_{n+1},t_{n+2+j},v^{n+2+j})\int_{t_{n+1+j}}^{t_{n+2+j}}(s-t_{n+1+j})ds-
        g(x,t_{n+1},t_{n+1+j},\right.
       \end{equation*}
       \begin{equation}\label{20}
       \left.v^{n+1+j})\int_{t_{n+1+j}}^{t_{n+2+j}}(s-t_{n+2+j})ds\right]+\frac{1}{2}\int_{t_{n+1+j}}^{t_{n+2+j}}(s-t_{n+1+j})(s-t_{n+2+j})\frac{\partial^{2}}{\partial s^{2}}[g(x,t_{n+1},\epsilon_{2}(s),v(\epsilon_{2}(s)))]ds.
       \end{equation}

       But the function $s\mapsto \frac{\partial^{2}}{\partial s^{2}}[g(x,t_{n+1},\epsilon_{2}(s),v(\epsilon_{2}(s)))]$ is integrable on $[t_{n+1+j},\text{\,}t_{n+2+j}]$ and the function $s\mapsto (s-t_{n+1+j}) (s-t_{n+2+j})$ doesn't change sign on this interval, applying the weighted mean value theorem and performing direct computations, equation $(\ref{20})$ becomes
       \begin{equation*}
        \int_{t_{n+1+j}}^{t_{n+2+j}}g(x,t_{n+1},s,v(s))ds=\frac{\sigma}{2}\left[g(x,t_{n+1},t_{n+2+j},v^{n+2+j})+g(x,t_{n+1},t_{n+1+j},v^{n+1+j})\right]+
       \end{equation*}
       \begin{equation*}
        \frac{1}{2}[g_{ss}(x,t_{n+1},\overline{\epsilon}_{2},v(\overline{\epsilon}_{2}))+(v_{s}(\overline{\epsilon}_{2})+
        v_{ss}(\overline{\epsilon}_{2}))g_{v}(x,t_{n+1},\overline{\epsilon}_{2},v(\overline{\epsilon}_{2}))
       +v_{s}^{2}(\overline{\epsilon}_{2})g_{vv}(x,t_{n+1},\overline{\epsilon}_{2},v(\overline{\epsilon}_{2}))]*
       \end{equation*}
       \begin{equation}\label{21}
       \int_{t_{n+1+j}}^{t_{n+2+j}}(s-t_{n+1+j})(s-t_{n+2+j})ds,
       \end{equation}
       where $\overline{\epsilon}_{2}\in(t_{n+1+j},t_{n+2+j})$ and "*" denotes the usual multiplication in $\mathbb{R}$. But, straightforward computations provide
       \begin{equation}\label{22}
       \int_{t_{n+1+j}}^{t_{n+2+j}}(s-t_{n+1+j})(s-t_{n+2+j})ds=-\frac{\sigma^{3}}{6}.
       \end{equation}

       Plugging equations $(\ref{21})$ and $(\ref{22})$, we obtain
       \begin{equation*}
        \int_{t_{n+1+j}}^{t_{n+2+j}}g(x,t_{n+1},s,v(s))ds=\frac{\sigma}{2}\left[g(x,t_{n+1},t_{n+2+j},v^{n+2+j})+g(x,t_{n+1},t_{n+1+j},v^{n+1+j})\right]-
       \end{equation*}
       \begin{equation}\label{23}
       \frac{\sigma^{3}}{12}[g_{ss}(x,t_{n+1},\overline{\epsilon}_{2},v(\overline{\epsilon}_{2}))+(v_{s}(\overline{\epsilon}_{2})+
        v_{ss}(\overline{\epsilon}_{2}))g_{v}(x,t_{n+1},\overline{\epsilon}_{2},v(\overline{\epsilon}_{2}))
       +v_{s}^{2}(\overline{\epsilon}_{2})g_{vv}(x,t_{n+1},\overline{\epsilon}_{2},v(\overline{\epsilon}_{2}))].
       \end{equation}

       Setting $z=\int_{t-\tau}^{t}g(x,t,s,v(s))ds$, combining equations $(\ref{23})$, $(\ref{18a})$ and $(\ref{13})$ and applying the Taylor series expansion for the function $f$, with respect to the fourth variable $z$, to get
       \begin{equation*}
       G(v^{n+1})=f\left(x,t_{n+1},v^{n+1},\frac{\sigma}{2}\overset{-1}{\underset{j=-m}\sum}\left(g(x,t_{n+1},t_{n+2+j},v^{n+2+j})+g(x,t_{n+1},t_{n+1+j},v^{n+1+j})\right)\right)-
       \end{equation*}
       \begin{equation}\label{24}
       \frac{m\sigma^{3}}{12}H(x,t_{n+1},\overline{\epsilon}_{2})\frac{\partial f}{\partial z}(x,t_{n+1},v^{n+1},\theta_{n}(x,\overline{\epsilon}_{2})).
       \end{equation}
       Here, 
       \begin{equation*}
       H(x,t_{n+1},\overline{\epsilon}_{2})=[g_{ss}(x,t_{n+1},\overline{\epsilon}_{2},v(\overline{\epsilon}_{2}))
       +(v_{s}(\overline{\epsilon}_{2})+v_{ss}(\overline{\epsilon}_{2}))g_{v}(x,t_{n+1},\overline{\epsilon}_{2},v(\overline{\epsilon}_{2}))\\
       +v_{s}^{2}(\overline{\epsilon}_{2})g_{vv}(x,t_{n+1},\overline{\epsilon}_{2},v(\overline{\epsilon}_{2}))],
       \end{equation*}
       $\theta_{n}(x,\overline{\epsilon}_{2})$ is between the minimum and maximum of $\frac{\sigma}{2}\overset{-1}{\underset{j=-m}\sum}(g(x,t_{n+1},t_{n+2+j},v^{n+2+j})+
      g(x,t_{n+1},t_{n+1+j},v^{n+1+j})$ and 
      \begin{equation*}
       \frac{\sigma}{2}\overset{-1}{\underset{j=-m}\sum}(g(x,t_{n+1},t_{n+2+j},v^{n+2+j})+g(x,t_{n+1},t_{n+1+j},v^{n+1+j}))-
       \frac{m\sigma^{3}}{12}H(x,t_{n+1},\overline{\epsilon}_{2}).
       \end{equation*}
       Let
       \begin{equation*}
        F(v^{n+1})=f\left(x,t_{n+1},v^{n+1},\frac{\sigma}{2}\overset{-1}{\underset{j=-m}\sum}\left(g(x,t_{n+1},t_{n+2+j},v^{n+2+j})+g(x,t_{n+1},t_{n+1+j},v^{n+1+j})\right)\right)=
       \end{equation*}
       \begin{equation}\label{25}
       f\left(x,t_{n+1},v^{n+1},\frac{\sigma}{2}\overset{m}{\underset{j=1}\sum}\left(g(x,t_{n+1},t_{n+2-j},v^{n+2-j})+g(x,t_{n+1},t_{n+1-j},v^{n+1-j})\right)\right),
       \end{equation}
       \begin{equation}\label{26}
       R_{1}^{n+1}=H(x,t_{n+1},\overline{\epsilon}_{2})\frac{\partial f}{\partial z}(x,t_{n+1},v^{n+1},\theta_{n}(x,\overline{\epsilon}_{2})).
       \end{equation}

       Since $\tau=m\sigma$, substituting equation $(\ref{24})$ into equation $(\ref{18})$, utilizing equations $(\ref{25})$-$(\ref{26})$ and rearranging terms, this results in
       \begin{equation*}
        3A(v^{n+1},u)+2\alpha\sigma a(v^{n+1},u)-2\sigma\left(F(v^{n+1}),u\right)_{0}=A(4v^{n}-v^{n-1},u)+\beta a_{\Gamma}(3\rho^{n+1}-4\rho^{n}+\rho^{n-1},u)+
       \end{equation*}
       \begin{equation}\label{27}
       2\alpha\sigma a_{\Gamma}(\rho^{n+1},u)-\frac{\sigma^{3}\tau}{6}\left(R_{1}^{n+1},u\right)_{0}-
       \frac{2\sigma^{3}}{3}\left((\mathcal{I}-\beta\Delta)v_{3t}(\epsilon_{1}(t)),u\right)_{0}.
       \end{equation}

       Omitting the error term: $-\frac{\sigma^{3}\tau}{6}\left(R_{1}^{n+1},u\right)_{0}-\frac{2\sigma^{3}}{3}\left((\mathcal{I}-\beta\Delta)v_{3t}(\epsilon_{1}(t)),u\right)_{0}$, and replacing the exact solution $v\in V$ with the approximate one $v_{h}(t)\in V_{h}$, for every $t\in[0,\text{\,}T_{f}]$, to obtain the desired high-order computational approach for solving the nonlinear Sobolev equation with distributed delay $(\ref{1})$, subjects to initial-boundary conditions $(\ref{2})$-$(\ref{3})$. That is, given $v^{n-1},v^{n}\in V_{h}$, find $v^{n+1}\in V_{h}$, for $n=1,2,...,N-1$, so that
       \begin{equation*}
        3A(v_{h}^{n+1},u)+2\alpha\sigma a(v_{h}^{n+1},u)-2\sigma\left(F(v_{h}^{n+1}),u\right)_{0}=A(4v_{h}^{n}-v_{h}^{n-1},u)+\beta a_{\Gamma}(3\rho^{n+1}-4\rho^{n}+\rho^{n-1},u)+
       \end{equation*}
       \begin{equation}\label{28}
       2\alpha\sigma a_{\Gamma}(\rho^{n+1},u),\text{\,\,\,\,\,\,\,\,}\forall u\in H^{1}(\Omega),
       \end{equation}
       subjects to initial condition
       \begin{equation}\label{29}
       v_{h}^{n}=u_{0}^{n},\text{\,\,\,\,for\,\,\,\,}n=-m,-m+1,...,0.
       \end{equation}

        Formulations $(\ref{28})$-$(\ref{29})$ represent the proposed high-order unconditionally stable computational technique in a computed solution of the initial-boundary value problem $(\ref{1})$-$(\ref{3})$.

     \section{Analysis of stability and error estimates}\label{sec3}
       This Section deals with the stability analysis and error estimates of the developed computational technique $(\ref{28})$-$(\ref{29})$ in a computed solution of the initial-boundary value problem $(\ref{1})$-$(\ref{3})$. We assume that the finite element space $V_{h}$ defined by equation $(\ref{5})$ satisfies the approximation properties of the piecewise polynomials of degrees $3$, $4$ and $5$. Moreover, for any $u\in H^{5}(\Omega)$,
     \begin{equation}\label{29a}
     \inf\{\|u-u_{h}\|_{0}^{2}:\text{\,\,}u_{h}\in V_{h}\}\leq C_{1}h^{10}\|u\|_{5}^{2}\text{\,\,\,\,and\,\,\,\,}\inf\{\|\nabla(u-u_{h})\|_{\bar{0}}^{2}:\text{\,\,}u_{h}\in V_{h}\}\leq C_{2}h^{8}\|u\|_{5}^{2},
     \end{equation}
     where $C_{j}$, for $j=1,2$, are two positive constants independent of the time step $\sigma$ and space size $h$. For the sake of stability analysis and error estimates, it is assumed that the analytical solution $v$ belongs to the Sobolev space $H^{4}(0,T_{f};\text{\,}H^{5})$, that is, there is a constant $\widehat{C}>0$, independent of the mesh size $h$ (maximum diameter of triangles $"T"$ of the triangulation $\mathcal{F}_{h}$) and time step $\sigma$, so that
      \begin{equation}\label{31a}
     \||v|\|_{5,4}\leq \widehat{C}.
     \end{equation}

     \begin{lemma}\label{l1}
       The bilinear operator $A(\cdot,\cdot)$ defined by equation $(\ref{17})$ is symmetric and positive definite on $V\times V$.
       \end{lemma}

       \begin{proof}
       Let $u,w\in V$, then it holds
       \begin{equation*}
        A(u,w)=\left(u,w\right)_{0}+\beta a(u,w)=\left(w,u\right)_{0}+\beta\left(\nabla u,\nabla w\right)_{\bar{0}}=\left(w,u\right)_{0}+\beta\left(\nabla w,\nabla u\right)_{\bar{0}}=\left(w,u\right)_{0}+\beta a(w,u)=A(w,u).
       \end{equation*}

       So, the operator $A(\cdot,\cdot)$ is symmetric. Additionally,
       \begin{equation*}
        A(u,u)=\left(u,u\right)_{0}+\beta a(u,u)=\|u\|_{0}^{2}+\beta\|\nabla u\|_{\bar{0}}^{2}>0,\text{\,\,\,\,for\,\,\,\,every\,\,\,\,}u\neq0.
       \end{equation*}
       Thus, $A(\cdot,\cdot)$ is positive definite. This ends the proof of Lemma $\ref{l1}$.
       \end{proof}

       It follows from Lemma $\ref{l1}$ that the bilinear operator $A(\cdot,\cdot)$ defines a scalar product on the Sobolev space $V\times V$. Let $\|\cdot\|_{\beta}$ be the corresponding norm. Thus
       \begin{equation}\label{30}
       \|u\|_{\beta}= \sqrt{\|u\|_{0}^{2}+\beta\|\nabla u\|_{\bar{0}}^{2}},\text{\,\,\,\,\,for\,\,\,\,any\,\,\,\,}u\in V.
       \end{equation}

       Using estimates $(\ref{29a})$ and $(\ref{30})$, it is not hard to observe that
     \begin{equation}\label{30a}
     \inf\{\|u-u_{h}\|_{\beta}^{2}:\text{\,\,}u_{h}\in V_{h}\}\leq C_{3}h^{8}\|u\|_{5}^{2},
     \end{equation}
     where $C_{3}>0$, is a constant that does not dependent on the time step $\sigma$ and space size $h$.

      \begin{lemma}\label{l2}
       The scalar product $A(\cdot,\cdot)$ defined by equation $(\ref{17})$ satisfies the following estimates. For every $u,w\in V$,
       \begin{equation}\label{31}
     A(u,w)\leq \sqrt{2}\|u\|_{\beta}\|w\|_{\beta},\text{\,\,\,\,and\,\,\,\,\,}A(u,u)\geq \min\{1,\text{\,}\beta\}\|u\|_{\beta}^{2}.
     \end{equation}
       \end{lemma}

       \begin{proof}
       Suppose that $u,w\in V$, using the Cauchy-Schwarz inequality, it holds
       \begin{equation*}
        A(u,w)=\left(u,w\right)_{0}+\beta a(u,w)=\left(u,w\right)_{0}+\beta\left(\nabla u,\nabla w\right)_{\bar{0}}\leq \|u\|_{0}\|w\|_{0}+\beta\|\nabla u\|_{\bar{0}}\|\nabla w\|_{\bar{0}}=
       \end{equation*}
       \begin{equation*}
        [(\|u\|_{0}\|w\|_{0}+\beta\|\nabla u\|_{\bar{0}}\|\nabla w\|_{\bar{0}})^{2}]^{\frac{1}{2}}\leq [2(\|u\|_{0}^{2}\|w\|_{0}^{2}+\beta^{2}\|\nabla u\|_{\bar{0}}^{2}\|\nabla w\|_{\bar{0}}^{2})]^{\frac{1}{2}}\leq
       \end{equation*}
       \begin{equation*}
        \sqrt{2} [(\|u\|_{0}^{2}+\beta\|\nabla u\|_{\bar{0}}^{2})(\|w\|_{0}^{2}+\beta\|\nabla w\|_{\bar{0}}^{2})]^{\frac{1}{2}}=\sqrt{2}\|u\|_{\beta}\|w\|_{\beta}.
       \end{equation*}

        This ends the proof of the first estimate in Lemma $\ref{l2}$. The proof of the second inequality in this Lemma is obvious according to the definition of the norm $\|\cdot\|_{\beta}$.
       \end{proof}

     \begin{theorem}\label{t1} (Stability analysis and error estimates).
      Let $v\in H^{4}(0,T_{f};\text{\,} H^{5})$, be the analytical solution of the nonlinear Sobolev equation $(\ref{1})$, subjects to initial-boundary conditions $(\ref{2})$-$(\ref{3})$, and $v_{h}:\text{\,}[0,\text{\,}T_{f}]\rightarrow V_{h}$, be the numerical solution provided by the new algorithm $(\ref{28})$-$(\ref{29})$. Set $e_{h}^{n}=v_{h}^{n}-v^{n}$, be the error term at the discrete time $t_{n}$. Thus, it holds
      \begin{equation*}
      \|v_{h}^{n+1}\|_{\beta}^{2}+2\sigma\underset{\geq0}{\left(\underbrace{-4\gamma_{1}C_{p}+\frac{101}{72}\alpha}\right)}\underset{k=1}{\overset{n}\sum}
      a(e_{h}^{k+1},e_{h}^{k+1})\leq 2\widehat{C}^{2}+2\left[10C_{3}\|v^{1}\|_{5}^{2}h^{8}+\right.
       \end{equation*}
       \begin{equation}\label{s1}
       \left.T_{f}\left(3^{-1}C_{g}\tau^{2}+4\||(\mathcal{I}-\beta\Delta)v_{3t}|\|_{0,\infty}^{2}\right)\sigma^{4}\right]\exp\left(\frac{8T_{f}
       C_{p}(\beta+C_{p})(\tau\gamma_{2}\gamma_{3})^{2}}{\alpha\beta}\right),
       \end{equation}
       \begin{equation*}
      \|e_{h}^{n+1}\|_{\beta}^{2}+\sigma\underset{\geq0}{\left(\underbrace{-4\gamma_{1}C_{p}+\frac{101}{72}\alpha}\right)}\underset{k=1}{\overset{n}\sum}a(e_{h}^{k+1},e_{h}^{k+1})\leq \left[10C_{3}\|v^{1}\|_{5}^{2}h^{8}+\right.
       \end{equation*}
       \begin{equation}\label{s2}
       \left.T_{f}\left(3^{-1}C_{g}\tau^{2}+4\||(\mathcal{I}-\beta\Delta)v_{3t}|\|_{0,\infty}^{2}\right)\sigma^{4}\right]\exp\left(\frac{8T_{f}
       C_{p}(\beta+C_{p})(\tau\gamma_{2}\gamma_{3})^{2}}{\alpha\beta}\right),
       \end{equation}
       for $n=1,2,...,N-1$, where $C_{g}$, $C_{p}$, $C_{3}$, and $\gamma_{i}$, for $i=1,2,3$, are positive constants independent of the time step $\sigma$ and mesh grid $h$, $\alpha$ and $\beta$ are the physical parameters given in equation $(\ref{1})$ whereas $\widehat{C}$ is the constant defined in estimate $(\ref{31a})$. It's worth mentioning that estimate $(\ref{s1})$ suggests that the computational algorithm $(\ref{28})$-$(\ref{29})$ for computing an approximate solution of the initial-boundary value problem $(\ref{1})$-$(\ref{3})$ is unconditionally stable while inequality $(\ref{s2})$ indicates that the numerical approach is temporal second-order accurate and spatial fourth-order convergent.
       \end{theorem}

      \begin{proof}
       Subtracting equation $(\ref{27})$ from equation $(\ref{28})$, and rearranging terms we obtain
       \begin{equation*}
        3A(e_{h}^{n+1},\text{\,}u)+2\alpha\sigma a(e_{h}^{n+1},u)=A(4e_{h}^{n}-e_{h}^{n-1},\text{\,}u)+2\sigma\left(F(v_{h}^{n+1})-F(v^{n+1}),u\right)_{0}+
       \end{equation*}
       \begin{equation*}
        \frac{\sigma^{3}\tau}{6}\left(R_{1}^{n+1},u\right)_{0}+\frac{2\sigma^{3}}{3}\left((\mathcal{I}-\beta\Delta)v_{3t}(\epsilon_{1}(t)),u\right)_{0},\text{\,\,\,\,}\forall u\in H^{1}(\Omega).
       \end{equation*}
        Taking $u=2e_{h}^{n+1}$, and rearranging terms this equation becomes
       \begin{equation*}
        2A(3e_{h}^{n+1}-4e_{h}^{n}+e_{h}^{n-1},\text{\,}e_{h}^{n+1})+4\alpha\sigma a(e_{h}^{n+1},e_{h}^{n+1})=4\sigma\left(F(v_{h}^{n+1})-F(v^{n+1}),e_{h}^{n+1}\right)_{0}+
       \end{equation*}
       \begin{equation}\label{32a}
        \frac{\sigma^{3}\tau}{3}\left(R_{1}^{n+1},e_{h}^{n+1}\right)_{0}+\frac{4\sigma^{3}}{3}\left((\mathcal{I}-\beta\Delta)v_{3t}(\epsilon_{1}(t)),e_{h}^{n+1}\right)_{0}.
       \end{equation}

       It follows from equation $(\ref{25})$ that
       \begin{equation*}
       F(v_{h}^{n+1})-F(v^{n+1})=f\left(x,t_{n+1},v_{h}^{n+1},\frac{\sigma}{2}\overset{m}{\underset{j=1}\sum}\left(g(x,t_{n+1},t_{n+2-j},v_{h}^{n+2-j})+
       g(x,t_{n+1},t_{n+1-j},v_{h}^{n+1-j})\right)\right)-
       \end{equation*}
       \begin{equation*}
       f\left(x,t_{n+1},v^{n+1},\frac{\sigma}{2}\overset{m}{\underset{j=1}\sum}\left(g(x,t_{n+1},t_{n+2-j},v^{n+2-j})+g(x,t_{n+1},t_{n+1-j},v^{n+1-j})\right)\right).
       \end{equation*}

       Since the functions $f$ and $g$ satisfy a Lipschitz condition in the third and fourth variables, there exist positive constants $\gamma_{j}$, for $j=1,2,3$, independent of the mesh grid $h$ and time step $\sigma$, so that
       \begin{equation*}
        \|F(v_{h}^{n+1})-F(v^{n+1})\|_{0}\leq \gamma_{1}\|v_{h}^{n+1}-v^{n+1}\|_{0}+\frac{\gamma_{2}\sigma}{2}\overset{m}{\underset{j=1}\sum}\|g(x,t_{n+1},t_{n+1-j},v_{h}^{n+1-j})-
        g(x,t_{n+1},t_{n+1-j},v^{n+1-j})+
       \end{equation*}
       \begin{equation*}
        g(x,t_{n+1},t_{n+2-j},v_{h}^{n+2-j})-g(x,t_{n+1},t_{n+2-j},v^{n+2-j})\|_{0}.
       \end{equation*}

       Applying the triangular inequality to get
       \begin{equation*}
        \|F(v_{h}^{n+1})-F(v^{n+1})\|_{0}\leq \gamma_{1}\|e_{h}^{n+1}\|_{0}+\frac{\gamma_{2}\sigma}{2}\overset{m}{\underset{j=1}\sum}[\|g(x,t_{n+1},t_{n+1-j},v_{h}^{n+1-j})-
        g(x,t_{n+1},t_{n+1-j},v^{n+1-j})\|_{0}+
       \end{equation*}
       \begin{equation*}
        \|g(x,t_{n+1},t_{n+2-j},v_{h}^{n+2-j})-g(x,t_{n+1},t_{n+2-j},v^{n+2-j})\|_{0}]\leq \gamma_{1}\|e_{h}^{n+1}\|_{0}+\frac{\gamma_{2}\gamma_{3}\sigma}{2}\overset{m}{\underset{j=1}\sum}[\|v_{h}^{n+1-j}-v^{n+1-j}\|_{0}+
       \end{equation*}
       \begin{equation}\label{32}
        \|v_{h}^{n+2-j}-v^{n+2-j}\|_{0}]\leq \gamma_{1}\|e_{h}^{n+1}\|_{0}+\frac{\gamma_{2}\gamma_{3}\sigma}{2}\overset{m}{\underset{j=1}\sum}(\|e_{h}^{n+1-j}\|_{0}+
        \|e_{h}^{n+2-j}\|_{0}).
       \end{equation}

       Furthermore, since the operator $A(\cdot,\cdot)$ is a scalar product, straightforward calculations provide
       \begin{equation*}
       A(3e_{h}^{n+1}-4e_{h}^{n}+e_{h}^{n-1},\text{\,}e_{h}^{n+1}+e_{h}^{n-1})=3A(e_{h}^{n+1}-e_{h}^{n},\text{\,}e_{h}^{n+1}+e_{h}^{n-1})-
       A(e_{h}^{n}-e_{h}^{n-1},\text{\,}e_{h}^{n+1}+e_{h}^{n-1})=
       \end{equation*}
       \begin{equation*}
       3[A(e_{h}^{n+1}-e_{h}^{n},\text{\,}e_{h}^{n+1}-e_{h}^{n})+A(e_{h}^{n+1}-e_{h}^{n},\text{\,}e_{h}^{n}+e_{h}^{n-1})]-
       [A(e_{h}^{n}-e_{h}^{n-1},\text{\,}e_{h}^{n+1}+e_{h}^{n})-A(e_{h}^{n}-e_{h}^{n-1},\text{\,}e_{h}^{n}-e_{h}^{n-1})]=
       \end{equation*}
       \begin{equation}\label{33}
       3\|e_{h}^{n+1}-e_{h}^{n}\|_{\beta}^{2}+\|e_{h}^{n}-e_{h}^{n-1}\|_{\beta}^{2}+3A(e_{h}^{n+1}-e_{h}^{n},\text{\,}e_{h}^{n}+e_{h}^{n-1})-
       A(e_{h}^{n}-e_{h}^{n-1},\text{\,}e_{h}^{n+1}+e_{h}^{n}).
       \end{equation}

        \begin{equation*}
       A(3e_{h}^{n+1}-4e_{h}^{n}+e_{h}^{n-1},\text{\,}e_{h}^{n+1}-e_{h}^{n-1})=3A(e_{h}^{n+1}-e_{h}^{n},\text{\,}e_{h}^{n+1}-e_{h}^{n-1})-
       A(e_{h}^{n}-e_{h}^{n-1},\text{\,}e_{h}^{n+1}-e_{h}^{n-1})=
       \end{equation*}
       \begin{equation*}
       3[A(e_{h}^{n+1}-e_{h}^{n},\text{\,}e_{h}^{n+1}-e_{h}^{n})+A(e_{h}^{n+1}-e_{h}^{n},\text{\,}e_{h}^{n}-e_{h}^{n-1})]-
       [A(e_{h}^{n}-e_{h}^{n-1},\text{\,}e_{h}^{n+1}-e_{h}^{n})+A(e_{h}^{n}-e_{h}^{n-1},\text{\,}e_{h}^{n}-e_{h}^{n-1})]=
       \end{equation*}
       \begin{equation}\label{34}
       3\|e_{h}^{n+1}-e_{h}^{n}\|_{\beta}^{2}-\|e_{h}^{n}-e_{h}^{n-1}\|_{\beta}^{2}+3A(e_{h}^{n+1}-e_{h}^{n},\text{\,}e_{h}^{n}-e_{h}^{n-1})-
       A(e_{h}^{n}-e_{h}^{n-1},\text{\,}e_{h}^{n+1}-e_{h}^{n}).
       \end{equation}

       A combination of equations $(\ref{33})$ and $(\ref{34})$ gives
       \begin{equation}\label{35}
        2A(3e_{h}^{n+1}-4e_{h}^{n}+e_{h}^{n-1},\text{\,}e_{h}^{n+1})=6\|e_{h}^{n+1}-e_{h}^{n}\|_{\beta}^{2}+6A(e_{h}^{n+1}-e_{h}^{n},\text{\,}e_{h}^{n})
       -2A(e_{h}^{n}-e_{h}^{n-1},\text{\,}e_{h}^{n+1}).
       \end{equation}

       But, $6A(e_{h}^{n+1}-e_{h}^{n},\text{\,}e_{h}^{n})=-3(\|e_{h}^{n+1}-e_{h}^{n}\|_{\beta}^{2}+\|e_{h}^{n}\|_{\beta}^{2}-\|e_{h}^{n+1}\|_{\beta}^{2})$ and
       $-2A(e_{h}^{n}-e_{h}^{n-1},\text{\,}e_{h}^{n+1})=-2A(e_{h}^{n}-e_{h}^{n+1},\text{\,}e_{h}^{n+1})-2A(e_{h}^{n+1}-e_{h}^{n-1},\text{\,}e_{h}^{n+1})=
       \|e_{h}^{n+1}-e_{h}^{n}\|_{\beta}^{2}-\|e_{h}^{n}\|_{\beta}^{2}-\|e_{h}^{n+1}-e_{h}^{n-1}\|_{\beta}^{2}+\|e_{h}^{n-1}\|_{\beta}^{2}$. This fact, together with equation $(\ref{35})$ yield
       \begin{equation*}
        2A(3e_{h}^{n+1}-4e_{h}^{n}+e_{h}^{n-1},\text{\,}e_{h}^{n+1})=4\|e_{h}^{n+1}-e_{h}^{n}\|_{\beta}^{2}-\|e_{h}^{n+1}-e_{h}^{n-1}\|_{\beta}^{2}+
        3(\|e_{h}^{n+1}\|_{\beta}^{2}-\|e_{h}^{n}\|_{\beta}^{2})-(\|e_{h}^{n}\|_{\beta}^{2}-\|e_{h}^{n-1}\|_{\beta}^{2}).
       \end{equation*}

       Since $-\|e_{h}^{n+1}-e_{h}^{n-1}\|_{\beta}^{2}\geq -2(\|e_{h}^{n+1}-e_{h}^{n}\|_{\beta}^{2}+\|e_{h}^{n}-e_{h}^{n-1}\|_{\beta}^{2})$, this equation implies
       \begin{equation}\label{36}
        2A(3e_{h}^{n+1}-4e_{h}^{n}+e_{h}^{n-1},\text{\,}e_{h}^{n+1})\geq 2(\|e_{h}^{n+1}-e_{h}^{n}\|_{\beta}^{2}-\|e_{h}^{n}-e_{h}^{n-1}\|_{\beta}^{2})+
        3(\|e_{h}^{n+1}\|_{\beta}^{2}-\|e_{h}^{n}\|_{\beta}^{2})-(\|e_{h}^{n}\|_{\beta}^{2}-\|e_{h}^{n-1}\|_{\beta}^{2}).
       \end{equation}

       For an arbitrary $\epsilon>0$, the application of the Cauchy-Schwarz inequality along with estimate $(\ref{32})$ provide
       \begin{equation*}
       4\sigma\left(F(v_{h}^{n+1})-F(v^{n+1}),e_{h}^{n+1}\right)_{0}\leq 4\sigma\|e_{h}^{n+1}\|_{0}\|F(v_{h}^{n+1})-F(v^{n+1})\|_{0}\leq
       4\gamma_{1}\sigma\|e_{h}^{n+1}\|_{0}^{2}+\epsilon\sigma\|e_{h}^{n+1}\|_{0}^{2}+
       \end{equation*}
       \begin{equation}\label{37}
        \frac{(\gamma_{2}\gamma_{3})^{2}\sigma^{3}}{\epsilon}\left[\overset{m}{\underset{j=1}\sum}(\|e_{h}^{n+1-j}\|_{0}+
        \|e_{h}^{n+2-j}\|_{0})\right]^{2}.
       \end{equation}

       It is not difficult to show (by mathematical induction) that
       \begin{equation}\label{38}
       \left[\overset{m}{\underset{j=1}\sum}(\|e_{h}^{n+1-j}\|_{0}+\|e_{h}^{n+2-j}\|_{0})\right]^{2}\leq 2m \overset{m}{\underset{j=1}\sum}(\|e_{h}^{n+1-j}\|_{0}^{2}+\|e_{h}^{n+2-j}\|_{0}^{2})\leq 4m\overset{m}{\underset{j=0}\sum}\|e_{h}^{n+1-j}\|_{0}^{2}.
       \end{equation}

       Because $\tau=n\sigma$, substituting estimate $(\ref{38})$ into estimate $(\ref{37})$, to obtain
       \begin{equation}\label{39}
       4\sigma\left(F(v_{h}^{n+1})-F(v^{n+1}),e_{h}^{n+1}\right)_{0}\leq (4\gamma_{1}+\epsilon)\sigma\|e_{h}^{n+1}\|_{0}^{2}+
        \frac{4\tau(\gamma_{2}\gamma_{3})^{2}\sigma^{2}}{\epsilon}\overset{m}{\underset{j=0}\sum}\|e_{h}^{n+1-j}\|_{0}^{2}.
       \end{equation}

       But it follows from the Poincar\'{e}-Friedrich inequality that for every $u\in H^{1}(\Omega)$, $\|u\|_{0}^{2}\leq C_{p} \|\nabla u\|_{\bar{0}}^{2}=C_{p}a(u,u)$, where $C_{p}>0$, is a constant independent of the time step $\sigma$ and grid space $h$. Thus,
       \begin{equation}\label{39a}
       \|u\|_{0}^{2}\leq \frac{1}{2}(\|u\|_{0}^{2}+C_{p} \|\nabla u\|_{\bar{0}}^{2})\leq \frac{\beta+C_{p}}{2\beta}(\|u\|_{0}^{2}+\beta\|\nabla u\|_{\bar{0}}^{2})=
       \frac{\beta+C_{p}}{2\beta}\|u\|_{\beta}^{2},
       \end{equation}
       where $\|\cdot\|_{\beta}$ denotes the norm associated with the scalar product $A(\cdot,\cdot)$ and defined by equation $(\ref{30})$. Utilizing estimate $(\ref{39a})$, estimate $(\ref{39})$ implies
       \begin{equation}\label{40}
       -4\sigma\left(F(v_{h}^{n+1})-F(v^{n+1}),e_{h}^{n+1}\right)_{0}\leq C_{p}(4\gamma_{1}+\epsilon)\sigma a(e_{h}^{n+1},e_{h}^{n+1})+
        \frac{2\tau(\beta+C_{p})(\gamma_{2}\gamma_{3})^{2}\sigma^{2}}{\epsilon\beta}\overset{m}{\underset{j=0}\sum}\|e_{h}^{n+1-j}\|_{\beta}^{2}.
       \end{equation}
        Furthermore,
       \begin{equation}\label{41}
        \frac{\sigma^{3}\tau}{3}\left(R_{1}^{n+1},e_{h}^{n+1}\right)_{0}\leq
        \frac{\epsilon\sigma}{12}\|e_{h}^{n+1}\|_{0}^{2}+\frac{\sigma^{5}\tau^{2}}{3}\|R_{1}^{n+1}\|_{0}^{2}\leq \frac{C_{p}\epsilon\sigma}{12}a(e_{h}^{n+1},e_{h}^{n+1})+\frac{\sigma^{5}\tau^{2}}{3}\|R_{1}^{n+1}\|_{0}^{2}.
       \end{equation}

       Because the functions $f$ and $g$ are regular enough, there is a positive constant $C_{g}$ which does not depend on the mesh grid $h$ and time step $\sigma$, so that
        \begin{equation*}
       \|R_{1}^{n+1}\|_{0}^{2}=\int_{\Omega}\left(H(x,t_{n+1},\overline{\epsilon}_{2})\frac{\partial f}{\partial z}(x,t_{n+1},v^{n+1},
       \theta_{n}(x,\overline{\epsilon}_{2}))\right)dx\leq C_{g},
       \end{equation*}
       where the term $R_{1}^{n+1}$ is given by equation $(\ref{26})$. This fact, combined with estimate $(\ref{41})$ result in
       \begin{equation}\label{42}
        \frac{\sigma^{3}\tau}{3}\left(R_{1}^{n+1},e_{h}^{n+1}\right)_{0}\leq \frac{C_{p}\epsilon\sigma}{12}a(e_{h}^{n+1},e_{h}^{n+1})+\frac{C_{g}\sigma^{5}\tau^{2}}{3}.
       \end{equation}

       But $v\in H^{4}(0,T_{f};\text{\,}H^{5})$, so $\Delta v_{3t}\in H^{1}(0,T_{f};\text{\,}H^{3})$, which implies $(\mathcal{I}-\Delta)v_{3t}\in H^{1}(0,T_{f};\text{\,}H^{3})\subset L^{\infty}(0,T_{f};\text{\,}L^{2})$. Hence,
       \begin{equation*}
       \|(\mathcal{I}-\Delta)v_{3t}(\epsilon_{1}(t))\|_{0}\leq \||(\mathcal{I}-\Delta)v_{3t}|\|_{0,\infty},\text{\,\,\,\,\,\,\,}\forall t\in[0,\text{\,}T_{f}],
       \end{equation*}
       where $\||\cdot|\|_{0,\infty}$ represents the norm on the space $L^{\infty}(0,T_{f};\text{\,}L^{2})$. Using this, it holds
       \begin{equation*}
        \frac{4\sigma^{3}}{3}\left((\mathcal{I}-\beta\Delta)v_{3t}(\epsilon_{1}(t)),e_{h}^{n+1}\right)_{0}\leq \frac{4\sigma^{3}}{3}\|(\mathcal{I}-\beta\Delta)v_{3t}(\epsilon_{1}(t))\|_{0}\|e_{h}^{n+1}\|_{0}\leq \frac{\epsilon\sigma}{9}\|e_{h}^{n+1}\|_{0}^{2}+
       \end{equation*}
        \begin{equation}\label{43}
       4\sigma^{5}\||(\mathcal{I}-\Delta)v_{3t}|\|_{0,\infty}^{2}\leq \frac{\epsilon C_{p}\sigma}{9}a(e_{h}^{n+1},e_{h}^{n+1})+ 4\sigma^{5}\||(\mathcal{I}-\Delta)v_{3t}|\|_{0,\infty}^{2}.
       \end{equation}

       Substituting estimates $(\ref{36})$, $(\ref{40})$, $(\ref{42})$ and $(\ref{43})$ into equation $(\ref{32a})$ and rearranging terms, it is not difficult to observe that
       \begin{equation*}
        3(\|e_{h}^{n+1}\|_{\beta}^{2}-\|e_{h}^{n}\|_{\beta}^{2})-(\|e_{h}^{n}\|_{\beta}^{2}-\|e_{h}^{n-1}\|_{\beta}^{2})+2(\|e_{h}^{n+1}-e_{h}^{n}\|_{\beta}^{2}-
        \|e_{h}^{n}-e_{h}^{n-1}\|_{\beta}^{2})+4\alpha\sigma a(e_{h}^{n+1},e_{h}^{n+1})\leq
       \end{equation*}
       \begin{equation}\label{44}
       C_{p}\left(4\gamma_{1}+\frac{43}{36}\epsilon\right)\sigma a(e_{h}^{n+1},e_{h}^{n+1})+
        \frac{2\tau(\beta+C_{p})(\gamma_{2}\gamma_{3})^{2}\sigma^{2}}{\epsilon\beta}\overset{m}{\underset{j=0}\sum}\|e_{h}^{n+1-j}\|_{\beta}^{2}+
       \left(\frac{C_{g}\tau^{2}}{3}+4\||(\mathcal{I}-\Delta)v_{3t}|\|_{0,\infty}^{2}\right)\sigma^{5}.
       \end{equation}

       Without loss of generality, we assume that the physical parameters $\alpha>0$ is chosen so that $\alpha>2\gamma_{1}C_{p}$. Set $\epsilon=\frac{\alpha}{C_{p}}$, so
       $\left(4\gamma_{1}+\frac{43}{36}\epsilon\right)C_{p}<4\alpha$. This fact, together with estimate $(\ref{44})$ imply
       \begin{equation*}
        3(\|e_{h}^{n+1}\|_{\beta}^{2}-\|e_{h}^{n}\|_{\beta}^{2})-(\|e_{h}^{n}\|_{\beta}^{2}-\|e_{h}^{n-1}\|_{\beta}^{2})+2(\|e_{h}^{n+1}-e_{h}^{n}\|_{\beta}^{2}-
        \|e_{h}^{n}-e_{h}^{n-1}\|_{\beta}^{2})+\underset{\geq0}{\left(\underbrace{-4\gamma_{1}C_{p}+\frac{101}{36}\alpha}\right)}\sigma a(e_{h}^{n+1},e_{h}^{n+1})\leq
       \end{equation*}
       \begin{equation*}
      \frac{2\tau C_{p}(\beta+C_{p})(\gamma_{2}\gamma_{3})^{2}\sigma^{2}}{\alpha\beta}\overset{m}{\underset{j=0}\sum}\|e_{h}^{n+1-j}\|_{\beta}^{2}+
       \left(\frac{C_{g}\tau^{2}}{3}+4\||(\mathcal{I}-\Delta)v_{3t}|\|_{0,\infty}^{2}\right)\sigma^{5}.
       \end{equation*}

       Summing this up from $k=1,...,n$, and rearranging terms, this provides
       \begin{equation*}
        3\|e_{h}^{n+1}\|_{\beta}^{2}-\|e_{h}^{n}\|_{\beta}^{2}+2\|e_{h}^{n+1}-e_{h}^{n}\|_{\beta}^{2}+\underset{\geq0}{\left(\underbrace{-4\gamma_{1}C_{p}+\frac{101}{36}\alpha}\right)}\sigma \overset{n}{\underset{k=1}\sum}a(e_{h}^{k+1},e_{h}^{k+1})\leq 3\|e_{h}^{1}\|_{\beta}^{2}-\|e_{h}^{0}\|_{\beta}^{2}+2\|e_{h}^{1}-e_{h}^{0}\|_{\beta}^{2}+
       \end{equation*}
       \begin{equation}\label{45}
       \frac{2\tau C_{p}(\beta+C_{p})(\gamma_{2}\gamma_{3})^{2}\sigma^{2}}{\alpha\beta}\overset{n}{\underset{k=1}\sum}\overset{m}{\underset{j=0}\sum}
       \|e_{h}^{k+1-j}\|_{\beta}^{2}+ n\left(\frac{C_{g}\tau^{2}}{3}+4\||(\mathcal{I}-\Delta)v_{3t}|\|_{0,\infty}^{2}\right)\sigma^{5}.
       \end{equation}

       It follows from the initial conditions $(\ref{2})$ and $(\ref{29})$ that: $v^{k}=u_{0}^{k}$ and $v_{h}^{k}=u_{0}^{k}$, for $k=-m,-m+1,...,0$. So, $e_{h}^{k}=0$, for $k=-m,-m+1,...,0$. Utilizing this fact, it is easy to see that
       \begin{equation*}
       \overset{n}{\underset{k=1}\sum}\overset{m}{\underset{j=0}\sum}\|e_{h}^{k+1-j}\|_{\beta}^{2}=\overset{n}{\underset{k=1}\sum}\|e_{h}^{k+1}\|_{\beta}^{2}+
       \overset{n}{\underset{k=1}\sum}\|e_{h}^{k}\|_{\beta}^{2}+\overset{n}{\underset{k=1}\sum}\|e_{h}^{k-1}\|_{\beta}^{2}+...+
       \overset{n}{\underset{k=1}\sum}\|e_{h}^{k+1-m}\|_{\beta}^{2}=\overset{n+1}{\underset{k=2}\sum}\|e_{h}^{k}\|_{\beta}^{2}+
       \end{equation*}
       \begin{equation}\label{46}
       \overset{n}{\underset{k=1}\sum}\|e_{h}^{k}\|_{\beta}^{2}+ \overset{n-1}{\underset{k=1}\sum}\|e_{h}^{k}\|_{\beta}^{2}+...+
       \overset{n+1-m}{\underset{k=1}\sum}\|e_{h}^{k}\|_{\beta}^{2}\leq (m+1)\overset{n+1}{\underset{k=1}\sum}\|e_{h}^{k}\|_{\beta}^{2},
       \end{equation}
       where the sum equals zero if the upper summation index is nonpositive. But
       \begin{equation}\label{46a}
       -\|e_{h}^{n}\|_{\beta}^{2}=-\|e_{h}^{n}-e_{h}^{n+1}+e_{h}^{n+1}\|_{\beta}^{2}\geq -2(\|e_{h}^{n+1}-e_{h}^{n}\|_{\beta}^{2}+\|e_{h}^{n+1}\|_{\beta}^{2}).
       \end{equation}

       Plugging inequalities $(\ref{46})$ and $(\ref{46a})$ into estimate $(\ref{45})$, this implies
       \begin{equation*}
        \|e_{h}^{n+1}\|_{\beta}^{2}+\underset{\geq0}{\left(\underbrace{-4\gamma_{1}C_{p}+\frac{101}{36}\alpha}\right)}\sigma \overset{n}{\underset{k=1}\sum}a(e_{h}^{k+1},e_{h}^{k+1})\leq 5\|e_{h}^{1}\|_{\beta}^{2}+
       \end{equation*}
       \begin{equation}\label{47}
       \frac{2\tau C_{p}(\beta+C_{p})(\gamma_{2}\gamma_{3})^{2}\sigma^{2}}{\alpha\beta}(m+1)\overset{n+1}{\underset{k=1}\sum}
       \|e_{h}^{k}\|_{\beta}^{2}+ n\left(\frac{C_{g}\tau^{2}}{3}+4\||(\mathcal{I}-\Delta)v_{3t}|\|_{0,\infty}^{2}\right)\sigma^{5}.
       \end{equation}

       But, $\sigma=\frac{\tau}{m}$ which is assumed to satisfy $\sigma=\frac{T_{f}}{N}$, then $(m+1)\sigma\leq2\tau$ and $n\sigma^{5}\leq T_{f}\sigma^{4}$. For small values of the time step $\sigma$, substituting this into estimate $(\ref{47})$ and applying the Gronwall inequality, this results in
       \begin{equation*}
        \|e_{h}^{n+1}\|_{\beta}^{2}+\underset{\geq0}{\left(\underbrace{-4\gamma_{1}C_{p}+\frac{101}{36}\alpha}\right)}\sigma \overset{n}{\underset{k=1}\sum}a(e_{h}^{k+1},e_{h}^{k+1})\leq \left[5\|e_{h}^{1}\|_{\beta}^{2}+ \left(\frac{C_{g}\tau^{2}}{3}+4\||(\mathcal{I}-\Delta)v_{3t}|\|_{0,\infty}^{2}\right)\sigma^{4}\right]*
       \end{equation*}
       \begin{equation}\label{48}
       \exp\left(\frac{8T_{f}C_{p}(\beta+C_{p})(\tau\gamma_{2}\gamma_{3})^{2}}{\alpha\beta}\right),
       \end{equation}
       where "*" means the usual multiplication in $\mathbb{R}$. Indeed: $(n+1)\sigma\leq2T_{f}$. Now, let $u\in V_{h}$ be an arbitrary function, where $V_{h}$ is the finite element space defined by equation $(\ref{5})$. The application of triangular inequality gives
        \begin{equation}\label{48a}
       \|e_{h}^{1}\|_{\beta}^{2}=\|(v_{h}^{1}-u)+(u-v^{1})\|_{\beta}^{2}\leq 2[\|v_{h}^{1}-u\|_{\beta}^{2}+\|u-v^{1}\|_{\beta}^{2}].
       \end{equation}

        Since $\inf\{\|v_{h}^{1}-u\|_{\beta}^{2}:\text{\,\,}u\in V_{h}\}=0$, applying the infimum over $u\in V_{h}$, in the left and right sides of estimate $(\ref{48a})$ and utilizing inequality $(\ref{30a})$, to get
       \begin{equation*}
       \|e_{h}^{1}\|_{\beta}^{2}\leq 2\inf\{\|u-v^{1}\|_{\beta}^{2}:\text{\,\,}u\in V_{h}\}\leq 2C_{3}h^{8}\|v^{1}\|_{5}^{2}.
       \end{equation*}

       Substituting this into estimate $(\ref{48})$, we obtain
       \begin{equation*}
        \|e_{h}^{n+1}\|_{\beta}^{2}+\underset{\geq0}{\left(\underbrace{-4\gamma_{1}C_{p}+\frac{101}{36}\alpha}\right)}\sigma \overset{n}{\underset{k=1}\sum}a(e_{h}^{k+1},e_{h}^{k+1})\leq \left[10C_{3}\|v^{1}\|_{5}^{2}h^{8}+ \left(\frac{C_{g}\tau^{2}}{3}+4\||(\mathcal{I}-\Delta)v_{3t}|\|_{0,\infty}^{2}\right)\sigma^{4}\right]*
       \end{equation*}
       \begin{equation}\label{49}
       \exp\left(\frac{8T_{f}C_{p}(\beta+C_{p})(\tau\gamma_{2}\gamma_{3})^{2}}{\alpha\beta}\right).
       \end{equation}

       Since $\|v_{h}^{n+1}\|_{\beta}-\|v^{n+1}\|_{\beta}\leq \|e_{h}^{n+1}\|_{\beta}$, performing straightforward calculations, inequality $(\ref{49})$ becomes
       \begin{equation*}
        \|v_{h}^{n+1}\|_{\beta}^{2}+2\underset{\geq0}{\left(\underbrace{-4\gamma_{1}C_{p}+\frac{101}{36}\alpha}\right)}\sigma \overset{n}{\underset{k=1}\sum}a(e_{h}^{k+1},e_{h}^{k+1})\leq 2\|v^{n+1}\|_{\beta}^{2}+2\left[10C_{3}\|v^{1}\|_{5}^{2}h^{8}+\right.
       \end{equation*}
       \begin{equation*}
       \left. \left(\frac{C_{g}\tau^{2}}{3}+4\||(\mathcal{I}-\Delta)v_{3t}|\|_{0,\infty}^{2}\right)\sigma^{4}\right]
       \exp\left(\frac{8T_{f}C_{p}(\beta+C_{p})(\tau\gamma_{2}\gamma_{3})^{2}}{\alpha\beta}\right),
       \end{equation*}
       for $n=1,2,...,N-1$. The proof of Theorem $\ref{t1}$ is completed thank to estimate $(\ref{31a})$.
       \end{proof}

      \section{Numerical experiments}\label{sec4}
       This Section presents some numerical results to demonstrate the applicability and efficiency of the proposed high-order computational technique $(\ref{28})$-$(\ref{29})$ for solving the nonlinear Sobolev equation $(\ref{1})$, subjects to initial and boundary conditions $(\ref{2})$ and $(\ref{3})$, respectively. To check the unconditional stability and accuracy of the the new algorithm we assume that the time step $\sigma=2^{-m}$, for $m=4,5,...,8$, and the space size $h\in\{2^{-m},\text{\,}m=2,3,...,6\}$, where $h=\max\{d(T),\text{\,\,}T\in\mathcal{T}_{h}\}$ and $\mathcal{T}_{h}$ is the triangulation of the domain $\overline{\Omega}$. Furthermore, we take the parameters $\alpha=\beta=1$, and we compute the error $e_{\theta}=v_{\theta}-v$, using the strong norm, $\|\cdot\|_{1,\infty}$. Specifically,
       \begin{equation*}
        \||e_{\theta}|\|_{1,\infty}=\underset{1\leq n\leq N}{\max}\|e_{\theta}^{n}\|_{1},
       \end{equation*}
       where $\theta\in\{h,\text{\,}\sigma\}$. In addition, the convergence order, $R(h)$, in space of the proposed strategy is estimated using the formula
       \begin{equation*}
        R(h)=\frac{\log\left(\frac{\||e_{2h}|\|_{1,\infty}}{\||e_{h}|\|_{1,\infty}}\right)}{\log(2)},
       \end{equation*}
        where $e_{2h}$ and $e_{h}$ are the spatial errors associated with the grid sizes $2h$ and $h$, respectively, while the temporal convergence rate, $R(\sigma)$, is calculated as follows
       \begin{equation*}
        R(\sigma)=\frac{\log\left(\frac{\||e_{2\sigma}|\|_{1,\infty}}{\||e_{\sigma}|\|_{1,\infty}}\right)}{\log(2)},
       \end{equation*}
        where $e_{\sigma}$ and $e_{2\sigma}$ denote the errors in time corresponding to time steps $\sigma$ and $2\sigma$, respectively. Finally, the numerical computations are carried out with the help of MATLAB R$2007b$.\\

       $\bullet$\textbf{ Example 1}. Let $\overline{\Omega}=[0,\text{\,}1]$, and $T_{f}=3$. We consider the following one-dimensional distributed delay nonlinear Sobolev
       model defined in \cite{tr} as
       \begin{equation*}
        \left\{
          \begin{array}{ll}
            v_{t}(x,t)-v_{xxt}(x,t)-v_{xx}(x,t)=v^{2}(x,t)-2\int_{t-1}^{t}v(x,s)ds+f(x,t), & \hbox{for $(x,t)\in(0,\text{\,}1)\times(0,\text{\,}3]$,} \\
            \text{\,}\\
            v(x,t)=e^{-x}\cos(\pi t), & \hbox{for $(x,t)\in[0,\text{\,}1]\times[-1,\text{\,}0]$,} \\
            \text{\,}\\
            v(0,t)=\cos(\pi t),\text{\,\,\,}v(1,t)=e^{-1}\cos(\pi t), & \hbox{for $t\in[0,\text{\,}3]$,} \\
          \end{array}
        \right.
       \end{equation*}
       where $f(x,t)=[\frac{4}{\pi}\sin(\pi t)-(1+e^{-x}\cos(\pi t))\cos(\pi t)]e^{-x}$. The analytical solution is given by $v(x,t)=e^{-x}\cos(\pi t)$. It is easy to see
       that the exact solution $v\in H^{4}(0,3;\text{\,}H^{5})$ and the function $f$ is regular enough.\\

       \textbf{Table $1$}$\label{T1}$.
          \begin{equation*}
          \begin{array}{c c}
          \text{\,developed approach,\,\,where\,\,}\sigma=2^{-7},\text{\,\,} & \text{\,new method,\,\,with\,\,}h=2^{-5},\text{\,\,}\\
           \begin{tabular}{cccc}
            \hline
            $h$ &  $\||e_{h}|\|_{1,\infty}$ & R($h$) & CPU (s)\\
             \hline
            $2^{-2}$ &  $2.1864\times10^{-3}$ &   --  &  0.7293   \\

            $2^{-3}$ &  $1.4855\times10^{-4}$ & 3.8795 & 1.7168  \\

            $2^{-4}$ &  $9.3600\times10^{-6}$ & 3.9883 & 4.3761   \\

            $2^{-5}$ &  $6.1370\times10^{-7}$ & 3.9309 & 12.2188 \\

            $2^{-6}$ &  $3.8314\times10^{-8}$ & 4.0016 & 36.1244 \\
            \hline
          \end{tabular} & \begin{tabular}{cccc}
            \hline
            $\sigma$ &  $\||e_{\sigma}|\|_{1,\infty}$ & R($\sigma$) & CPU (s)\\
             \hline
            $2^{-4}$ &   $1.0361\times10^{-2}$ &   -- &  0.6871 \\

            $2^{-5}$ &   $2.6403\times10^{-3}$ & 1.9724 & 1.5521  \\

            $2^{-6}$ &   $7.1668\times10^{-4}$ & 1.8813 & 3.8444 \\

            $2^{-7}$ &   $1.7939\times10^{-4}$ & 1.9982 & 10.3393  \\

            $2^{-8}$ &   $4.4270\times10^{-5}$ & 2.0187 & 29.3895 \\
            \hline
          \end{tabular}
          \end{array}
          \end{equation*}

          \begin{equation*}
          \begin{array}{c c}
          \text{\,method\,developed\,in\,\,}\cite{tr},\text{\,\,}\sigma=2^{-15},\text{\,\,} &
           \text{\,method\,\,proposed\,\,in\,\,}\cite{tr},\text{\,\,}h=2^{-10},\text{\,\,}\\
          \begin{tabular}{ccc}
            \hline
            $h$ &  $\||e_{h}|\|_{1,\infty}$ & R($h$) \\
             \hline
            $2^{-1}$ &   $1.7437\times10^{-5}$ &   --- \\

            $2^{-2}$ &   $1.1395\times10^{-6}$ & 3.9357 \\

            $2^{-3}$ &   $7.1520\times10^{-8}$ & 3.9939  \\

            $2^{-4}$ &   $4.4719\times10^{-9}$ & 3.9994 \\

            $2^{-5}$ &   $2.8000\times10^{-10}$ & 3.9974 \\
            \hline
          \end{tabular} & \begin{tabular}{ccc}
            \hline
            $\Delta t$ &  $\||e_{\Delta t}|\|_{1,\infty}$ & R($\Delta t$) \\
             \hline
            $2^{-2}$ &   $9.9431\times10^{-3}$ &   -- - \\

            $2^{-3}$ &   $2.4995\times10^{-3}$ & 1.9920  \\

            $2^{-4}$ &   $6.2384\times10^{-4}$ & 2.0024 \\

            $2^{-5}$ &   $1.5584\times10^{-4}$ & 2.0011  \\

            $2^{-6}$ &   $3.8928\times10^{-5}$ & 2.0012 \\
            \hline
          \end{tabular}
          \end{array}
          \end{equation*}

          \textbf{Table $2$}$\label{T2}$.
          \begin{equation*}
          \begin{array}{c c}
          \text{\,proposed algorithm,\,\,where\,\,}\sigma=2^{-8} & \text{\,new technique,\,\,with\,\,}h=2^{-6}\\
           \begin{tabular}{cccc}
            \hline
            $h$ &  $\||e_{h}|\|_{0,\infty}$ & R($h$) & CPU (s)\\
             \hline
            $2^{-2}$ &  $5.1208\times10^{-3}$ &   --  &  0.6973   \\

            $2^{-3}$ &  $3.2261\times10^{-4}$ & 3.9885 & 1.5661  \\

            $2^{-4}$ &  $2.0521\times10^{-5}$ & 3.9746 & 3.8209   \\

            $2^{-5}$ &  $1.2799\times10^{-6}$ & 4.0030 & 10.2874 \\

            $2^{-6}$ &  $7.7419\times10^{-8}$ & 4.0472 & 28.9548 \\
            \hline
          \end{tabular} & \begin{tabular}{cccc}
            \hline
            $\sigma$ &  $\||e_{\sigma}|\|_{0,\infty}$ & R($\sigma$) & CPU (s)\\
             \hline
            $2^{-4}$ &   $3.1176\times10^{-3}$ &   -- &  0.5689  \\

            $2^{-5}$ &   $7.8449\times10^{-4}$ & 1.9906 & 1.2455 \\

            $2^{-6}$ &   $1.9612\times10^{-4}$ & 2.0000 & 2.9874  \\

            $2^{-7}$ &   $4.9608\times10^{-5}$ & 1.9831 & 7.5502  \\

            $2^{-8}$ &   $1.2334\times10^{-5}$ & 2.0079 & 20.7924  \\
            \hline
          \end{tabular}
          \end{array}
          \end{equation*}

          \begin{equation*}
          \begin{array}{c c}
          \text{\,method\,developed\,in\,\,}\cite{tr},\text{\,\,}\sigma=2^{-15},\text{\,\,}& \text{\,method\,\,proposed\,\,in\,\,}\cite{tr},\text{\,\,}h=2^{-10}\\
          \begin{tabular}{ccc}
            \hline
            $h$ &  $\||e_{h}|\|_{0,\infty}$ & R($h$) \\
             \hline
            $2^{-1}$ &   $2.4660\times10^{-5}$ &   --- \\

            $2^{-2}$ &   $1.5511\times10^{-6}$ & 3.9908   \\

            $2^{-3}$ &   $9.7113\times10^{-8}$ & 3.9976   \\

            $2^{-4}$ &   $6.1116\times10^{-9}$ & 3.9900   \\

            $2^{-5}$ &   $3.8350\times10^{-10}$ & 3.9943 \\
            \hline
          \end{tabular} & \begin{tabular}{ccc}
            \hline
            $\Delta t$ &  $\||e_{\Delta t}|\|_{0,\infty}$ & R($\Delta t$) \\
             \hline
            $2^{-2}$ &   $1.3692\times10^{-2}$ &   --- \\

            $2^{-3}$ &   $3.4398\times10^{-3}$ & 1.9929  \\

            $2^{-4}$ &   $8.5869\times10^{-4}$ & 2.0021 \\

            $2^{-5}$ &   $2.1443\times10^{-4}$ & 2.0016 \\

            $2^{-6}$ &   $5.3562\times10^{-5}$ & 2.0012  \\
            \hline
          \end{tabular}
          \end{array}
          \end{equation*}
        Tables $1$-$2$ indicate that the new approach $(\ref{28})$-$(\ref{29})$ is more efficient than the method discussed in \cite{tr}.
        \text{\,}\\
          \text{\,}\\
        $\bullet$\textbf{ Example 2}. Suppose that $\Omega=(0,\text{\,}1)^{2}$, and $[0,\text{\,}T_{f}]=[0,\text{\,}1]$. We consider the following two-dimensional
        distributed delay nonlinear Sobolev problem defined in \cite{tr} as
       \begin{equation*}
        \left\{
          \begin{array}{ll}
            v_{t}-\Delta v_{t}-\Delta v=\frac{1}{2}v^{2}+\sin(v)+\int_{t-1}^{t}v(x,y,s)ds+g(x,y,t), & \hbox{on $\Omega\times(0,\text{\,}1]$,} \\
            \text{\,}\\
            v(x,y,t)=e^{-\frac{t}{2}}\sin(\pi x)\sin(\pi y), & \hbox{for $(x,y,t)\in\overline{\Omega}\times[-1,\text{\,}0]$,} \\
            \text{\,}\\
            v(0,y,t)=v(1,y,t)=v(x,0,t)=v(x,1,t)=0, & \hbox{for $t\in[0,\text{\,}1]$,} \\
          \end{array}
        \right.
       \end{equation*}
       where $g(x,y,t)=[\frac{3}{2}-2e^{\frac{1}{2}}+\pi^{2}-\frac{1}{2}e^{-\frac{t}{2}}\sin(\pi x)\sin(\pi y)]e^{-\frac{t}{2}}\sin(\pi x)\sin(\pi y)-
      \sin\left(e^{-\frac{t}{2}}\sin(\pi x)\sin(\pi y)\right)$. The exact solution is defined as: $v(x,y,t)=e^{-\frac{t}{2}}\sin(\pi x)\sin(\pi y)$. Additionally, it is easy
      to see that the exact solution $v\in H^{4}(0,1;\text{\,}H^{5}(\Omega))$ and the function $g$ is smooth enough.\\

       \textbf{Table $3$}$\label{T3}$.
          \begin{equation*}
          \begin{array}{c c}
          \text{\,developed approach,\,\,where\,\,}\sigma=2^{-5},\text{\,\,} & \text{\,new method,\,\,with\,\,}h=2^{-4},\text{\,\,}\\
           \begin{tabular}{cccc}
            \hline
            $h$ &  $\||e_{h}|\|_{1,\infty}$ & R($h$) & CPU (s)\\
             \hline
            $2^{-2}$ &  $4.0012\times10^{-2}$ &   --  &  0.9548   \\

            $2^{-3}$ &  $2.5218\times10^{-3}$ & 3.9879 & 2.2391  \\

            $2^{-4}$ &  $1.5729\times10^{-4}$ & 4.0030 & 5.7370  \\

            $2^{-5}$ &  $9.8231\times10^{-6}$ & 4.0011 & 15.845 \\

            $2^{-6}$ &  $6.0507\times10^{-7}$ & 4.0210 & 47.2418 \\
            \hline
          \end{tabular} & \begin{tabular}{cccc}
            \hline
            $\sigma$ &  $\||e_{\sigma}|\|_{1,\infty}$ & R($\sigma$) & CPU (s)\\
             \hline
            $2^{-4}$ &   $1.3497\times10^{-3}$ &   -- &  0.9296 \\

            $2^{-5}$ &   $3.4428\times10^{-4}$ & 1.9710 & 2.1134  \\

            $2^{-6}$ &   $9.2254\times10^{-5}$ & 1.8999 & 5.1526  \\

            $2^{-7}$ &   $2.3198\times10^{-5}$ & 1.9916 & 13.8465  \\

            $2^{-8}$ &   $5.7947\times10^{-6}$ & 2.0012 & 39.4514  \\
            \hline
          \end{tabular}
          \end{array}
          \end{equation*}

          \begin{equation*}
          \begin{array}{c c}
          \text{\,method\,developed\,in\,\,}\cite{tr},\text{\,\,}\sigma=2^{-12},\text{\,\,} &
           \text{\,method\,\,proposed\,\,in\,\,}\cite{tr},\text{\,\,}h=2^{-6},\text{\,\,}\\
          \begin{tabular}{ccc}
            \hline
            $h$ &  $\||e_{h}|\|_{1,\infty}$ & R($h$) \\
             \hline
            $2^{-1}$ &   $1.4027\times10^{-2}$ &   --- \\

            $2^{-2}$ &   $8.7271\times10^{-4}$ & 4.0066  \\

            $2^{-3}$ &   $5.4527\times10^{-5}$ & 4.0005  \\

            $2^{-4}$ &   $3.4056\times10^{-6}$ & 4.0010 \\

            $2^{-5}$ &   $2.1027\times10^{-7}$ & 4.0175 \\
            \hline
          \end{tabular} & \begin{tabular}{ccc}
            \hline
            $\Delta t$ &  $\||e_{\Delta t}|\|_{1,\infty}$ & R($\Delta t$)\\
             \hline
            $2^{-3}$ &   $5.7944\times10^{-4}$ &   --- \\

            $2^{-4}$ &   $1.5246\times10^{-4}$ & 1.9262 \\

            $2^{-5}$ &   $3.9037\times10^{-5}$ & 1.9655  \\

            $2^{-6}$ &   $9.8648\times10^{-6}$ & 1.9845 \\

            $2^{-7}$ &   $2.4706\times10^{-6}$ & 1.9974 \\
            \hline
          \end{tabular}
          \end{array}
          \end{equation*}

          \textbf{Table $4$}$\label{T4}$.
          \begin{equation*}
          \begin{array}{c c}
          \text{\,proposed algorithm,\,\,where\,\,}\sigma=2^{-6} & \text{\,new technique,\,\,with\,\,}h=2^{-5}\\
           \begin{tabular}{cccc}
            \hline
            $h$ &  $\||e_{h}|\|_{0,\infty}$ & R($h$) & CPU (s)\\
             \hline
            $2^{-2}$ &  $3.5009\times10^{-2}$ &   --  &  0.9341   \\

            $2^{-3}$ &  $2.1563\times10^{-3}$ & 4.0211 & 2.1311  \\

            $2^{-4}$ &  $1.3477\times10^{-4}$ & 4.0000 & 5.1544   \\

            $2^{-5}$ &  $8.2389\times10^{-6}$ & 4.0319 & 13.7050 \\

            $2^{-6}$ &  $4.7998\times10^{-7}$ & 4.1014 & 39.3828 \\
            \hline
          \end{tabular} & \begin{tabular}{cccc}
            \hline
            $\sigma$ &  $\||e_{\sigma}|\|_{0,\infty}$ & R($\sigma$) & CPU (s)\\
             \hline
            $2^{-4}$ &   $1.1438\times10^{-3}$ &   -- &  0.9015  \\

            $2^{-5}$ &   $2.9044\times10^{-4}$ & 1.9775 & 2.0471  \\

            $2^{-6}$ &   $7.3440\times10^{-5}$ & 1.9836 & 4.8784  \\

            $2^{-7}$ &   $1.8768\times10^{-5}$ & 1.9683 & 12.6411  \\

            $2^{-8}$ &   $4.6886\times10^{-6}$ & 2.0017 & 34.6237  \\
            \hline
          \end{tabular}
          \end{array}
          \end{equation*}

          \begin{equation*}
          \begin{array}{c c}
          \text{\,method\,developed\,in\,\,}\cite{tr},\text{\,\,}\sigma=2^{-12},\text{\,\,} &
           \text{\,method\,\,proposed\,\,in,\,\,}\cite{tr},\text{\,\,}h=2^{-6},\text{\,\,}\\
          \begin{tabular}{ccc}
            \hline
            $h$ &  $\||e_{h}|\|_{0,\infty}$ & R($h$) \\
             \hline
            $2^{-1}$ &   $3.4021\times10^{-3}$ &   --- \\

            $2^{-2}$ &   $1.9640\times10^{-4}$ & 4.1145 \\

            $2^{-3}$ &   $1.2048\times10^{-5}$ & 4.0271  \\

            $2^{-4}$ &   $7.4896\times10^{-7}$ & 4.0076  \\

            $2^{-5}$ &   $4.6191\times10^{-8}$ & 4.0192 \\
            \hline
          \end{tabular} & \begin{tabular}{ccc}
            \hline
            $\Delta t$ &  $\||e_{\Delta t}|\|_{0,\infty}$ & R($\Delta t$)\\
             \hline
            $2^{-3}$ &   $1.2725\times10^{-4}$ &   --- \\

            $2^{-4}$ &   $3.3481\times10^{-5}$ & 1.9262  \\

            $2^{-5}$ &   $8.5727\times10^{-6}$ & 1.9655 \\

            $2^{-8}$ &   $2.1664\times10^{-6}$ & 1.9845 \\

            $2^{-7}$ &   $5.4255\times10^{-7}$ & 1.9974 \\
            \hline
          \end{tabular}
          \end{array}
          \end{equation*}
        Tables $3$-$4$ indicate that the constructed approach $(\ref{28})$-$(\ref{29})$ is more efficient than the method discussed in \cite{tr}.\\

        It follows from \textbf{Figures} $\ref{figure1}$-$\ref{figure2}$ and \textbf{Tables} $1$-$4$ that the developed algorithm $(\ref{28})$-$(\ref{29})$ is unconditionally
        stable, temporal second-order convergent and spatial fourth-order accurate. These computational results confirm the theoretical studies provided in Theorems $\ref{t1}$.

     \section{General conclusions and future works}\label{sec5}
     In this paper, we have constructed a high-order computational technique which is based on a combination of finite element formulation and interpolation approach for
     solving a nonlinear distributed delay Sobolev model. Both stability and error estimates of the proposed algorithm have been analyzed using a strong norm which is
     equivalent to the $L^{\infty}(0,T_{f};\text{\,}H^{1})$-norm. The theory suggested that the developed numerical method is unconditionally stable, spatial fourth-order
     accurate and temporal second-order convergent. The theoretical results have been confirmed by two numerical examples. Furthermore, both theoretical and computational
     studies have shown that the new algorithm is easy to implement and more efficient than a broad range of numerical methods discussed in the literature
     \cite{11tr,tr,21tr,33tr,25tr} for solving a general class of delay Sobolev equations. Our future works will develop a high-order numerical method in
     an approximate solution of distributed-order time fractional two-dimensional Sobolev problems.

     \subsection*{Ethical Approval}
     Not applicable.
     \subsection*{Availability of supporting data}
     Not applicable.
     \subsection*{Declaration of Interest Statement}
     The author declares that he has no conflict of interests.
     \subsection*{Funding}
     This work was supported and funded by the Deanship of Scientific Research at Imam Mohammad Ibn Saud Islamic University (IMSIU) (grant number IMSIU-DDRSP-NS25).
     \subsection*{Authors' contributions}
     The whole work has been carried out by the author.

    \newpage

        \begin{figure}
         \begin{center}
        Stability and convergence of the proposed high-order computational technique, $\alpha=\beta=1$
         \begin{tabular}{c c}
         \psfig{file=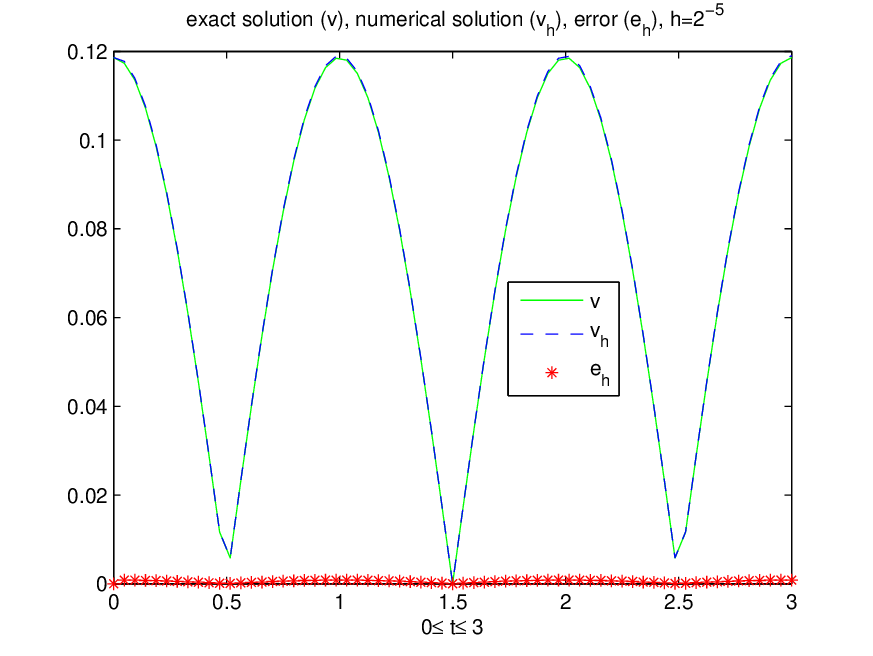,width=6cm} & \psfig{file=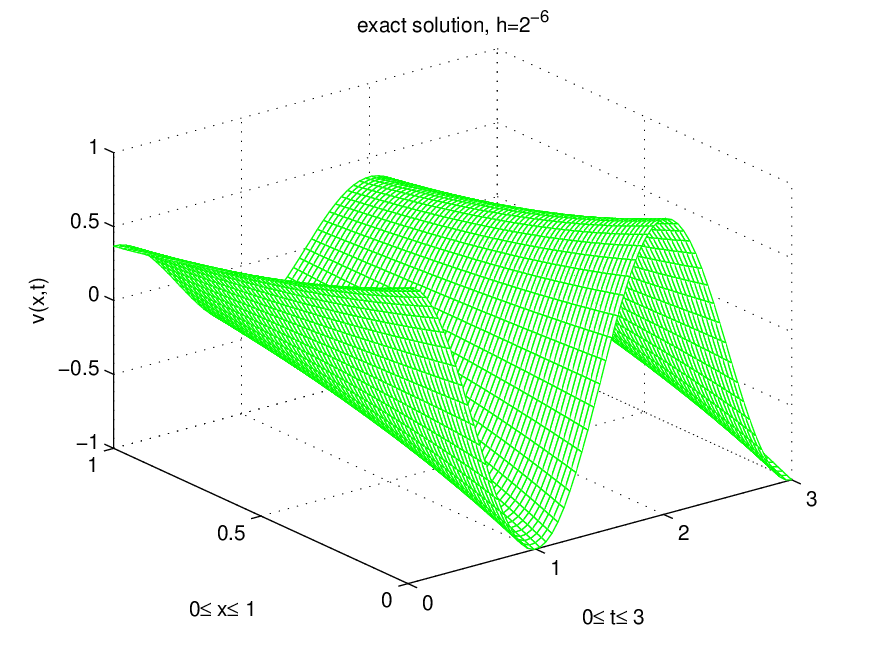,width=6cm}\\
         \psfig{file=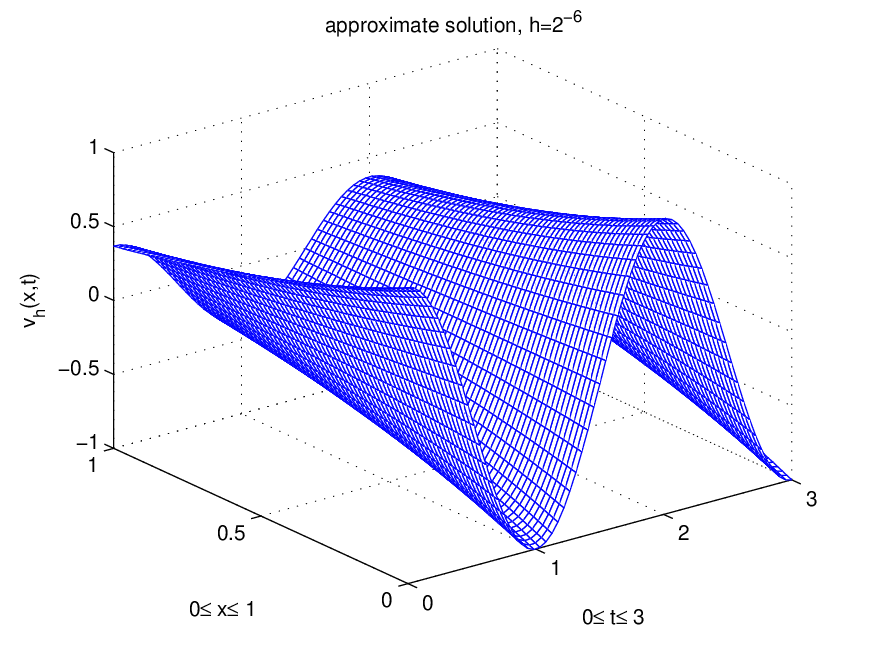,width=6cm} & \psfig{file=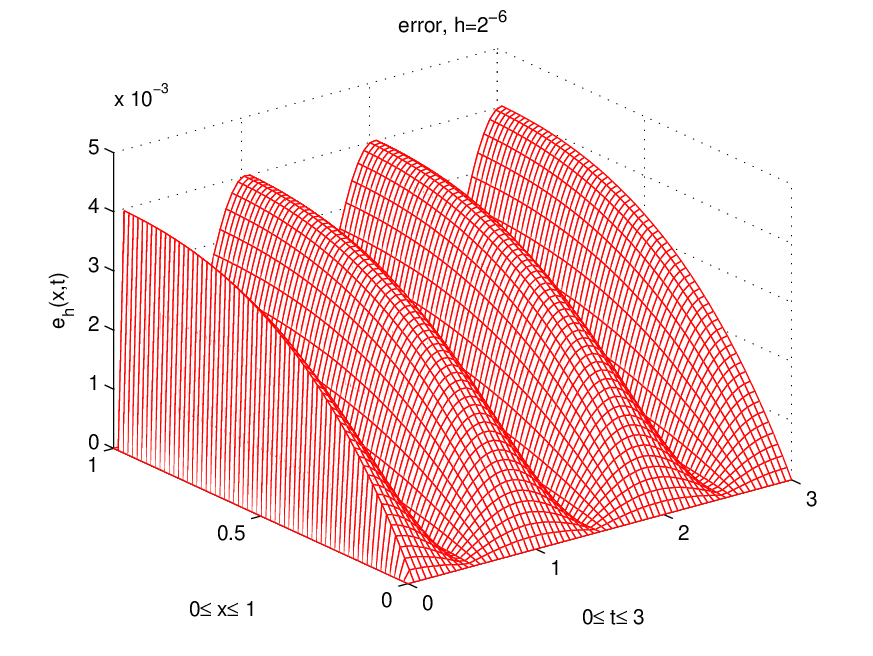,width=6cm}\\
         \end{tabular}
        \end{center}
        \caption{exact solution, numerical one and error, corresponding to \textbf{Example 1}}
        \label{figure1}
        \end{figure}

           \begin{figure}
         \begin{center}
         Stability analysis and convergence of the developed approach, $\alpha=\beta=1$
         \begin{tabular}{c c}
         \psfig{file=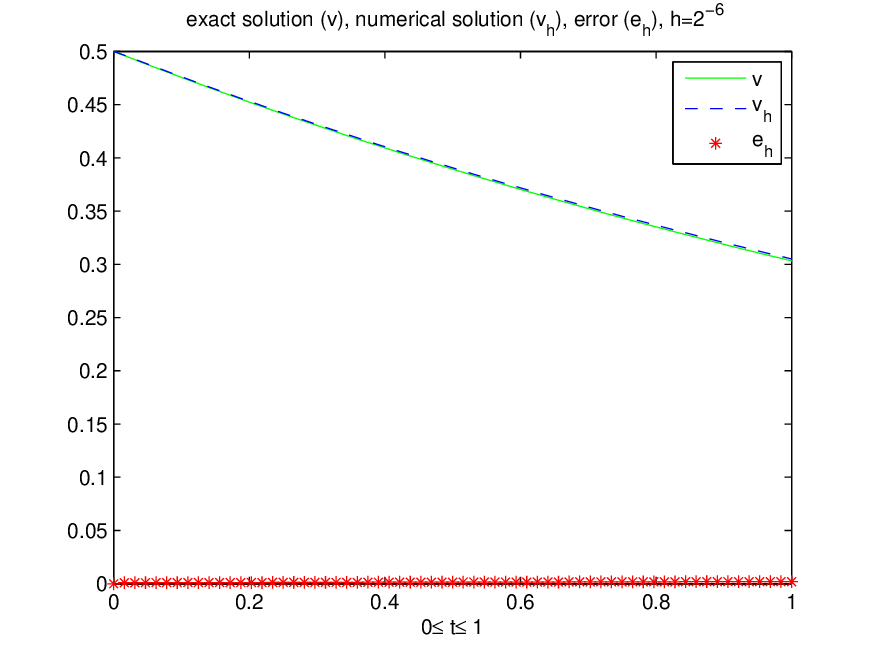,width=6cm} & \psfig{file=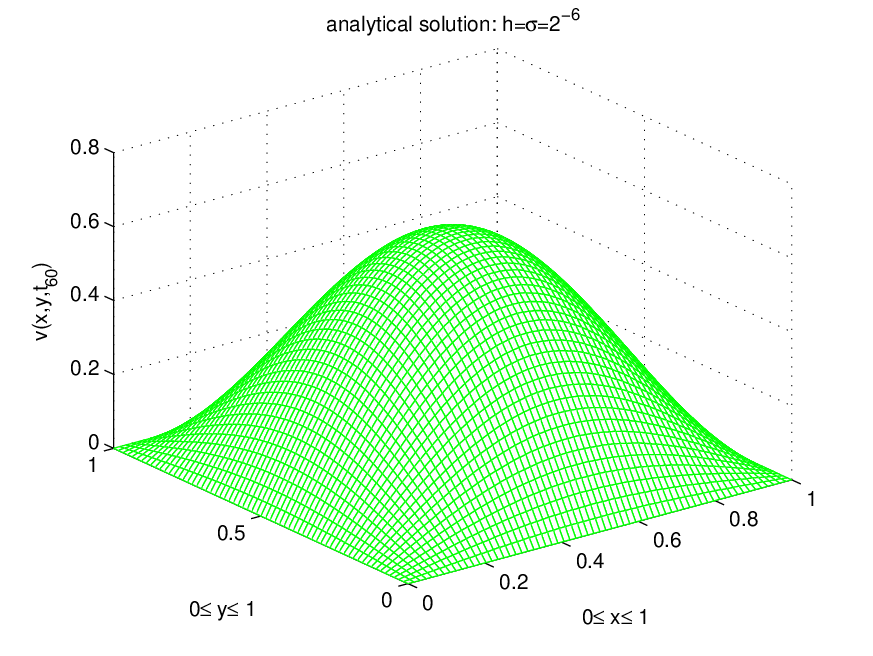,width=6cm}\\
         \psfig{file=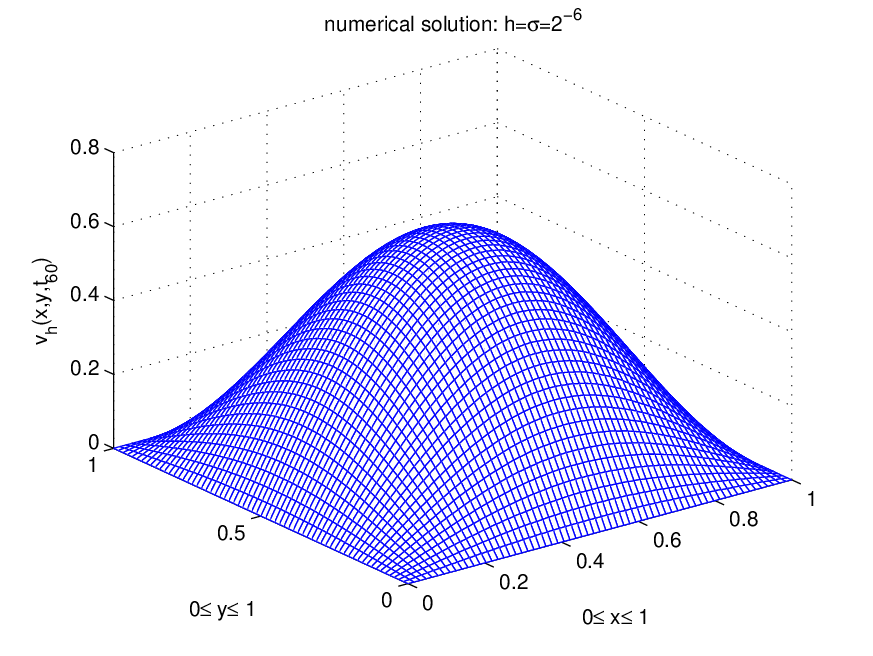,width=6cm} & \psfig{file=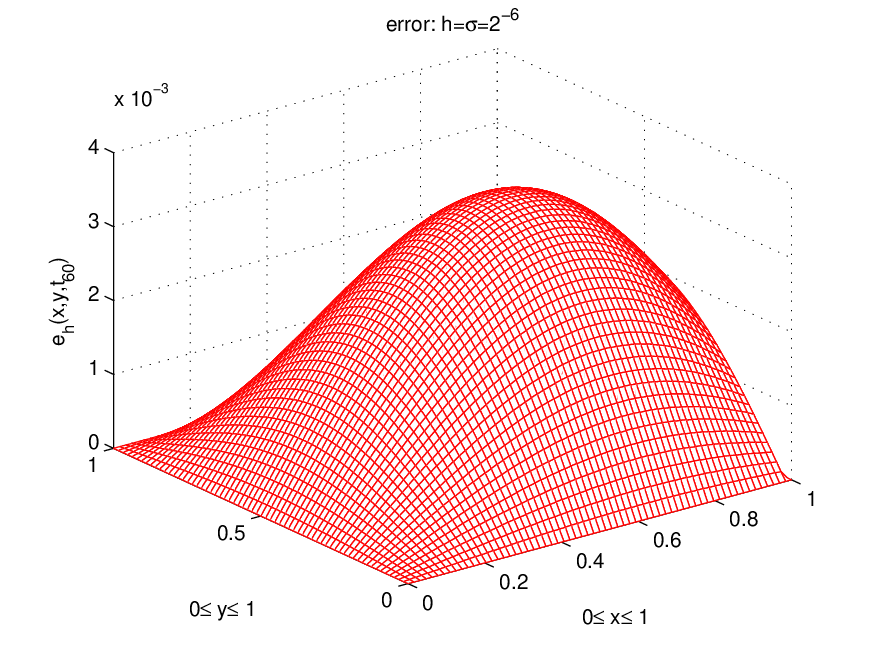,width=6cm}\\
         \end{tabular}
        \end{center}
        \caption{analytical solution, approximate one and error, associated with \textbf{Example 2}}
        \label{figure2}
        \end{figure}


\begin{thebibliography}{99}

    \bibitem{21tr}
    G. M. Amiraliyev, E. Cimen, I. Amirali, M. Cakir. "High-order finite difference technique for delay pseudo-parabolic equations", J. Comput. Appl. Math., $321 (2017)$, $1$-$7$.

    \bibitem{4xg}
    G. Chandrasekar, S. M. Boulaaras, S. Murugaiah, A. J. Gnanaprakasam, B. B. Cherif. "Analysis of a predator-prey model with distributed delay", J. Funct. Spaces, $2021 (2021)$, $6$.

    \bibitem{11tr}
    X. Chen, J. Duan, D. Li. "A Newton linearized compact finite difference scheme for one class of Sobolev equations", Numer. Methods Part. Differ. Equ., $34 (2018)$, $1093$-$1112$.

    \bibitem{22tr}
    A. B. Chiyaneh, H. Duru. "On adaptive mesh for the initial boundary value singularly perturbed delay Sobolev problems", Numer. Methods Part. Differ. Equ., $36 (2020)$, $228$-$248$.

    \bibitem{2tr}
    B. D. Coleman and W. Noll. "An approximation theorem for functionals, with applications to continuum mechanics", Arch. National Mech. Anal., $6 (1960)$, $355$-$370$.

    \bibitem{31tr}
    D. Deng. "The study of a fourth-order multistep ADI method applied to nonlinear delay reaction-diffusion equations", Appl. Numer. Math., $96 (2015)$, $118$-$133$.

    \bibitem{19tr}
    H. Di, Y. Shang, X. Zheng. "Global well-posedness for a fourth order pseudo-parabolic equation with memory and source terms", Discrete Contin. Dyn. Syst. Ser. B, $21 (2016)$, $781$-$801$.

    \bibitem{1tr}
    R. E. Ewing. "A coupled non-linear hyperbolic-Sobolev system", Ann. Mat. Pur. Appl., $114 (1977)$, $331$-$349$.

    \bibitem{12tr}
    R. E. Ewing. "Time-stepping Galerkin methods for nonlinear Sobolev partial differential equations", SIAMJ. Numer. Anal., $15 (1978)$, $1125$-$1150$.

    \bibitem{14tr}
    F. Gao, J. Qiu, Q. Zhang. "Local discontinuous Galerkin finite element method and error estimates for one class of Sobolev equation", J. Sci. Comput., $41 (2009)$, $436$-$460$.

    \bibitem{13xg}
    Z. M. Guo, X. M. Zhang. "Multiplicity results for periodic solutions to a class of second order delay differential equations", Commun. Pure Appl. Anal., $9 (2010)$, $1529$-$1542$.

    \bibitem{4tr}
    R. Huilgol. "A second order fluid of the differential type", Int. J. Nonlinear Mech., $3 (1968)$, $471$-$482$.

    \bibitem{17tr}
    J. H. Lightbourne, S.M. Rankin. "A partial functional differential equation of Sobolev type", J. Math.Anal. Appl., $93 (1983)$, $328$-$337$.

    \bibitem{19xg}
    G. J. Lin, R. Yuan. "Periodic solution for a predator-prey system with distributed delay", Math. Comput. Modelling, $42(2005)$, $959$-$966$.

    \bibitem{16xg}
    E. Kaslik, E. A. Kokovics, A. R\~{a}dulescu. "Stability and bifurcations in Wilson-Cowan systems with distributed delays,and an application to basal ganglia interactions", Commun. Nonlinear Sci. Numer. Simul., $104 (2022)$, $105984$.

    \bibitem{22xg}
    Q. Liu, D. Q. Jiang, T. Hayat. "Dynamics of stochastic predator-prey models with distributed delay and stage structure for prey", Int. J. Biomath., $14 (2021)$, $2150020$.

    \bibitem{21xg}
    J. Liu, J. B. Guan, Z. S. Feng. "Hopf bifurcation analysis of KdV-Burgers-Kuramoto chaotic system with distributed delay feedback", Internat. J. Bifur. Chaos, $29 (2019)$, $1950011$.

    \bibitem{en1}
     E. Ngondiep. "A robust time-split linearized explicit/implicit technique for solving a two-dimensional hydrodynamic model: Case of floods in the far north region of Cameroon", Phys. Fluids, $37(3)$ $(2025)$, $032101$-$18$.

    \bibitem{en2}
    E. Ngondiep. "An efficient numerical approach for solving three-dimensional Black-Scholes equation with stochastic volatility", Math. Meth. Appl. Sci., $(2024)$, $1$-$21$, Doi: $10.1002$/mma.$10576$.

     \bibitem{en3}
     E. Ngondiep."An efficient high-order weak Galerkin finite element approach for Sobolev equation with variable matrix coefficients", Comput. Math. Appl., $180$ $(2025)$ $279$-$298$.

    \bibitem{en4}
    E. Ngondiep. "An efficient modified Lax-Wendroff/interpolation approach with finite element method for a three-dimensional system of tectonic stress and deformation model: example of landslides in Cameroon", Phys. Fluids, $37(7)$ $(2025)$, DOI: $10.1063/5.0268529$.

    \bibitem{33tr}
    H. Qin, F. Wu, J. Zhang, C. Mu. "A linearized compact ADI scheme for semilinear parabolic problems with distributed delay", J. Sci. Comput., $87 (2021)$, $25$.

    \bibitem{18tr}
    Y. Shang, B. Guo. "On the problem of the existence of global solutions for a class of nonlinear convolutional integro-differential equations of pseudo parabolic type", Acta Math. Appl. Sin., $26 (2003)$, $511$-$524$.

    \bibitem{16tr}
    D. Shi, J.Wang, F. Yan. "Unconditional superconvergence analysis of $H^{1}$-Galerkin mixed finite element method for nonlinear Sobolev equations", Numer. Methods Part. Differ. Equ., $34 (2018)$, $145$-$166$.

    \bibitem{15tr}
    D. Shi, F. Yan, J. Wang. "Unconditional superconvergence analysis of a new mixed finite element method for nonlinear Sobolev equation", Appl. Math. Comput., $274 (2016)$, $182$-$194$.

    \bibitem{10tr}
    T. Sun, D. Yang. "The finite difference streamline diffusion methods for Sobolev equations with convection dominated term", Appl. Math. Comput., $125 (2002)$, $325$-$345$.

    \bibitem{tr}
    Z. Tan, M. Ran. "Linearized compact difference methods for solving nonlinear Sobolev equations with distributed delay", Numer. Methods Part. Differ. Equ., $39(2023)$, $2141$-$2162$.

    \bibitem{7tr}
    T. W. Ting. "A cooling process according to two-temperature theory of heat conduction", J. Math. Anal. Appl., $45(1974)$, $23$-$31$.

    \bibitem{29xg}
    N. Wang, M. A. Han. "Relaxation oscillations in predator-prey model with distributed delay", Comput. Appl. Math., $37(2018)$, $475$-$484$.

    \bibitem{36tr}
    F. Wu, X. Cheng, D. Li, J. Duan. "A two-level linearized compact ADI scheme for two-dimensional nonlinear reaction-diffusion equations", Comput. Math. Appl., $75 (2018)$, $2835$-$2850$.

    \bibitem{32tr}
    J. Xie, Z. Zhang. "The high-order multistep ADI solvers for two-dimensional nonlinear delayed reaction-diffusion equations with variable coefficients", Comput. Math. Appl., $75 (2018)$, $3558$-$3570$.

    \bibitem{31xg}
    J. S. Yu. "Modeling mosquito population suppression based on delay differential equations", SIAM J. Appl. Math., $78 (2018)$, $3168$-$3187$.

    \bibitem{25tr}
    C. Zhang, Z. Tan. "Linearized compact difference methods combined with Richardson extrapolation for nonlinear delay Sobolev equations", Commun. Nonlinear Sci. Numer. Simul., $91 (2020)$, $105461$.

    \bibitem{27tr}
    Q. Zhang, C. Zhang. "A new linearized compact multisplitting scheme for the nonlinear convection-reaction-diffusion equations with delay", Commun. Nonlinear Sci. Numer. Simul., $18 (2013)$, $3278$-$3288$.

    \bibitem{29tr}
    Q. Zhang, C. Zhang, L.Wang. "The compact and Crank-Nicolson ADI schemes for two-dimensional semilinear multidelay parabolic equations", J. Comput. Appl. Math., $306 (2016)$, $217$-$230$.

    \bibitem{35xg}
    B. Zheng, J. Li, J. S. Yu. "Existence and stability of periodic solutions in a mosquito population suppression model with time delay", J. Differential Equations, $315 (2022)$, $159$-$178$.

     \end{thebibliography}
       \end{document}